\documentclass[12pt]{article}
\usepackage[utf8]{inputenc}

\usepackage[english]{babel}
\usepackage{enumitem}
\usepackage{tikz-cd}
\usepackage{mathtools}
\usepackage{mathrsfs}
\usepackage{amsthm}
\usepackage{amssymb}
\usepackage{amsmath}
\usepackage{wasysym}
\usepackage{hyperref}
\usepackage{graphicx}
\usepackage{float}
\usepackage{framed}
\usepackage{fancybox}
\usepackage{geometry}
 \usepackage[all]{xy}

\usepackage{sectsty}
\sectionfont{\centering}

\usepackage{tikz}
\usetikzlibrary{
  knots,
  hobby,
  decorations.pathreplacing,
  shapes.geometric,
  calc
}

\tikzset{
  knot diagram/every strand/.append style={
    ultra thick,
    red
  },
  show curve controls/.style={
    postaction=decorate,
    decoration={show path construction,
      curveto code={
        \draw [blue, dashed]
        (\tikzinputsegmentfirst) -- (\tikzinputsegmentsupporta)
        node [at end, draw, solid, red, inner sep=2pt]{};
        \draw [blue, dashed]
        (\tikzinputsegmentsupportb) -- (\tikzinputsegmentlast)
        node [at start, draw, solid, red, inner sep=2pt]{}
        node [at end, fill, blue, ellipse, inner sep=2pt]{}
        ;
      }
    }
  },
  show curve endpoints/.style={
    postaction=decorate,
    decoration={show path construction,
      curveto code={
        \node [fill, blue, ellipse, inner sep=2pt] at (\tikzinputsegmentlast) {}
        ;
      }
    }
  }
}

\newcommand{\C}{\mathbb{C}}
\newcommand{\Z}{\mathbb{Z}}
\newcommand{\N}{\mathbb{N}}

\DeclareMathOperator{\id}{id}

\DeclareMathOperator{\Hom}{Hom}

\DeclareMathOperator{\Imm}{Im}

\DeclareMathOperator{\Aut}{Aut}
\mathchardef\hyp="2D

\DeclareMathOperator{\lk}{lk}

\DeclareMathOperator{\Meyer}{Meyer}
\DeclareMathOperator{\Burau}{\overline{\mathscr{B}}}
\DeclareMathOperator{\sgn}{sgn}
\DeclareMathOperator{\Isotr}{\mathbf{Isotr}}
\DeclareMathOperator{\Tangles}{\mathbf{Tangles}}

\newcommand{\F}{\mathscr{F}}
\DeclareMathOperator{\Ann}{Ann}

\DeclareMathOperator{\sign}{sign}
\DeclareMathOperator{\Maslov}{Maslov}

\newtheorem{thm}{Theorem}[section]

\newtheorem{lem}[thm]{Lemma}
\newtheorem{prop}[thm]{Proposition}

\theoremstyle{definition}
\newtheorem{dfn}[thm]{Definition}
\newtheorem{rmk}[thm]{Remark}

\newtheorem{ex}[thm]{Example}

\title{The homomorphism defect of an extended Levine-Tristram signature via twisted homology}
\author{Alice Merz}
\date{}

\begin{document}

\maketitle

\begin{abstract}
    Taking the Levine-Tristram signature of the closure of a braid defines a map from the braid group to the integers. A formula of Gambaudo and Ghys provides an evaluation of the homomorphism defect of this map in terms of the Burau representation and the Meyer cocycle. In 2017 Cimasoni and Conway generalized this formula to the multivariable signature of the closure of coloured tangles. In the present paper, we extend even further their result by using a different 4-dimensional interpretation of the signature. We obtain an evaluation of the additivity defect in terms of the Maslov index and the isotropic functor $\mathscr{F}_\omega$. We also show that in the case of coloured braids this defect can be rewritten in terms of the Meyer cocycle and the coloured Gassner representation, making it a direct generalization of the formula of Gambaudo and Ghys.
\end{abstract}

\section{Introduction}
Let $\mathcal{I}$ be a link invariant  which takes values in an abelian group $G$. One can consider the map
\begin{align*}
    B_n &\to G \\
    \alpha &\mapsto \mathcal{I}(\widehat{\alpha})
\end{align*}
where $B_n$ denotes the $n$-braid group and $ \widehat{\alpha}$ denotes the closure of the braid $\alpha$. A first question one can ask is whether this map is an homomorphism. It is not difficult to check, however, that the only link invariant with this property is the constant trivial one. It is therefore interesting to try to evaluate the homomorphism defect
$$\mathcal{I}(\widehat{\alpha\beta}) -  \mathcal{I}(\widehat{\alpha}) - \mathcal{I}(\widehat{\beta}).$$

In the case of the Levine-Tristram signatures (\cite{Levine},\cite{Tristram}), which we introduce in Subsection \ref{subsection: LT signature}, this study was carried out by Gambaudo and Ghys in \cite{GambaudoGhys}. 
We denote the Levine-Tristram signature of a link $L$ at $\omega \in S^1 \subset \C$ as
$\sigma_\omega(L). $
Gambaudo and Ghys were able to express the homomorphism defect of these signatures in terms of two classical objects, which are the reduced Burau representation (\cite{BirmanCannon}, \cite{Burau})
$$\Burau_t: B_n \to \operatorname{GL}_{n-1}(\Z[t,t^{-1}]) $$
and the Meyer cocycle \cite{Meyer}.

When we evaluate at some $\omega \in S^1 \setminus \{1\}$, for $\alpha$ an $n$-braid, $\Burau_\omega(\alpha) \in \operatorname{GL}_{n-1}(\C)$ is known to be unitary with respect to a skew-Hermitian form. 
Therefore, for $\alpha, \beta \in B_n$, one can consider the Meyer cocycle of the two unitary matrices $\Burau_\omega(\alpha)$ and $\Burau_\omega(\beta)$. 
The main consequence of the result in \cite{GambaudoGhys} is the equality
\begin{equation}
\label{GG}
\sigma_\omega(\widehat{\alpha\beta})- \sigma_\omega(\widehat{\alpha}) - \sigma_\omega(\widehat{\beta})= -\Meyer\left(\Burau_\omega(\alpha), \Burau_\omega(\beta)\right)
\end{equation}
for any pair of $n$-braids $\alpha, \beta$, for any $\omega \in S^1$ of order coprime to $n$.

The Levine-Tristram signature admits a generalization  which we introduce in Subsection \ref{secondo}. This generalization is called multivariate signature \cite{CimasoniConway} and it associates to every $\mu$-coloured link $L$ and every $\omega \in (S^1)^\mu \subset \C^\mu$ an integer $\sigma_\omega(L)$.

The Burau representation has a multivariable extension as well, called reduced coloured Gassner representation (discussed in \cite{ConwayPhD}, Chapter 9)
which is also unitary and one can wonder if the equality \eqref{GG} holds in this context as well.
This program was carried out by Cimasoni and Conway in \cite{CimasoniConway}, where they actually extend this result to oriented coloured tangles. Tangles no longer form a group but they represent the morphisms of a category. Moreover, tangles which are endomorphisms of a given object in this category can be composed and their closure gives a well-defined oriented link. Hence it makes sense to try to evaluate the additivity defect of the multivariate signature of tangles.
In order to obtain their equality, they make use of an interpretation of the multivariate signature as a $G$-signature of a certain branched cover of the four ball, which is valid only for $\omega \in (S^1 \setminus \{1\})^\mu$ of finite order. 

However, as pointed out by Conway in \cite{ConwayPhD}, one can use a different interpretation of the multivariate signature via twisted homology studied by \cite{ConwayNagelToffoli} which is valid for all $\omega$'s, to extend their theorem to a wider subset of $(S^1)^\mu$. We introduce twisted homology and state the interpretation of the multivariate signature in Subsection \ref{twistedhomology}. We state Novikov-Wall additivity theorem in Subsection \ref{sec-2}, which is a useful tool to calculate twisted signatures.
It turns out that the Gassner representation evaluated at a certain $\omega$ can be extended to coloured tangles as a functor $\mathscr{F}_\omega$, defined in Section \ref{section: isotropic functor} of the present paper, from the category of oriented coloured tangles to a so-called isotropic category. This functor is an extension of the Gassner representation in the sense that whenever $\alpha$ is a coloured braid, $\mathscr{F}_\omega(\alpha)$ will be the graph of $\Burau_\omega(\alpha)$. The image by $\mathscr{F}_\omega$ of a tangle provides a totally isotropic subspace of a complex vector space endowed with a skew-Hermitian form, therefore one cannot consider the Meyer cocycle of pairs of these spaces, but it is possible to consider the Maslov index of three such objects. 

In the present paper we will follow Conway's suggestion in \cite{ConwayPhD} (Remark 15.0.2) and use the twisted homology interpretation of the multivariate signature studied in \cite{ConwayNagelToffoli} to extend the equality of Cimasoni and Conway. 

Let $\tau$ be an oriented tangle in $D^2 \times [0,1]$ such that the endpoints of the strings in $\tau$ are contained in $ D^2 \times \{0,1\}$. We suppose that the endpoints in $D^2 \times \{0\}$ are fixed $\{p_1,\ldots, p_n\}$, possibly the empty set. Analogously, suppose that the endpoints in $D^2 \times \{1\}$ are fixed $\{p'_{1},\ldots, p'_{n'}\}$. The tangle $\tau$ is $\mu$\emph{-coloured} if each connected component of $\tau$ is labelled by $\{1,2,\ldots, \mu\}.$ The colouring of the strings and loops induces a labelling on the points $p_k$ and we assign $c(k)=+j$ or $-j$ where the sign is determined uniquely by the induced orientation on the point as the boundary of an oriented string of $\tau$. Similarly, the labelling of the strings of $\tau$ induces a colouring on the points $p'_{k'}$ and we assign $c'(k')=+j'$ or $-j'$ in the same way. We call $c:\{1,\ldots,n\} \to \{\pm 1, \ldots , \pm \mu\}$ the \emph{bottom colour} and $c':\{1,\ldots,n'\} \to \{\pm 1, \ldots , \pm \mu\}$ the \emph{top colour} of $\tau$, and call $\tau$ a $(c,c')$\emph{-tangle}. For fixed $j$ and $j'$, let $i_j$ and $i'_{j'}$ be the sum of the signs of $c(k)=\pm j$ and $c'(k')=\pm j'$, respectively. Remark that in this situation $i_j=i'_{j'}.$

With this notation in mind, we obtain the following generalization:

\begin{thm}\label{uno}
 Let $\tau_1$ and $\tau_2$ be $(c,c)$-tangles. Then we have
\begin{equation} \label{equa}
\sigma_\omega(\widehat{\tau_1\tau_2}) -\sigma_\omega(\widehat{\tau}_1)-\sigma_\omega(\widehat{\tau}_2) = \Maslov(\F_\omega(\overline{\tau}_1), \F_\omega(\id_c), \F_\omega(\tau_2)) 
\end{equation}
for all $\omega=(\omega_1,\ldots,\omega_\mu)\in (S^1\setminus\{1\})^\mu$ such that $\prod_{j=1}^\mu \omega_j^{i_j} \neq 1.$
Here $\overline{\tau}$ denotes the reflection of the tangle $\tau$ across a horizontal plane with opposite orientation and $\id_c$ denotes the trivial $n$-braid seen as a $(c,c)$-tangle.
\end{thm}

Coloured braids are isomorphisms in the category of coloured tangles, and braids that are $(c,c)$-tangles form a group $B_c$.
In the case of coloured braids, we will show that Equation \eqref{equa} provides exactly the multivariate extension of Equation \eqref{GG}:

\begin{thm} \label{due}
Given coloured braids $\alpha, \beta\in B_c$ and $\omega \in (S^1 \setminus \{1\})^\mu$ such that $\prod_{j=1}^\mu \omega_j^{i_j} \neq 1$ we have
\begin{equation*} 
\sigma_\omega(\widehat{\alpha \beta}) -\sigma_\omega(\widehat{\alpha})-\sigma_\omega(\widehat{\beta}) = -\Meyer(\Burau_{\omega}(\alpha), \Burau_{\omega}(\beta)).
\end{equation*}
\end{thm}

Here $\Burau_{(t_1,\ldots,t_\mu)}$ is the reduced coloured Gassner representation studied by Conway in \cite{ConwayPhD}.
\\ \\
The paper is organized as follows. Section \ref{sec-1} contains the definitions of most of the objects involved, such as the multivariate signature, and the statement of a result by Conway, Nagel and Toffoli \cite{ConwayNagelToffoli} which allows us to interpret the multivariate signatures as twisted signatures of $4$-dimensional manifolds.
We also briefly state Wall's result on the non-additivity of signatures of $4k$-manifolds, which can be found in \cite{Wall}. This is a useful tool to calculate twisted signatures of $4$-manifolds. 
In Section \ref{sec-3} we introduce our generalization of the Burau map to coloured tangles, which we call the isotropic functor.
Finally, in Section \ref{sec-4} we prove Theorems \ref{uno} and \ref{due}
\\ \\
\textit{Unless otherwise stated, all manifolds are assumed to be smooth.}
\\ \\
This paper is based on the author's Master thesis \cite{Merz}.

\section{Some basic definitions}
\label{sec-1}

\subsection{Levine-Tristram signature}
\label{subsection: LT signature}

Given an oriented link $L$, a Seifert surface for $L$ is a compact, connected, oriented surface $F$, embedded in $S^3$, which has $L$ as oriented boundary. 

Since $F$ is orientable, it admits a regular neighbourhood in $S^3$ homeomorphic to $F \times [-1,1]$ in which $F$ corresponds to $F \times \{0\}$. For $\varepsilon = \pm 1$ there are push-off maps 
$$i_\varepsilon: H_1 (F;\Z ) \to H_1 (S^3 \setminus F;\Z ) $$
defined by sending the homology class of a curve $ \gamma$ to the homology class of $\gamma \times \{ \varepsilon \}$.

There is a pairing, called the Seifert form:
\begin{align*}
    H_1 (F;\Z ) \times H_1 (F;\Z ) &\to \Z \\
    (a,b) & \mapsto \lk(a, i_+(b)).
\end{align*}

A matrix $A$ for this Seifert form is called a Seifert matrix. For every $\omega \in S^1 \subset \C$ observe that the matrix 
$$(1-\omega) A + (1-\overline{\omega})A^T $$ 
is Hermitian.

\begin{dfn}
Let $L$ be an oriented link, let $F$ be a Seifert surface for $L$ and let $A$ be a matrix representing the Seifert pairing of $F$. Given $\omega \in S^1$, the Levine-Tristram signature of L at $\omega$ is defined as the signature of $$(1-\omega) A + (1-\overline{\omega})A^T.$$
\end{dfn}

These signatures are well-defined link isotopy invariants, that is to say they are independent of the chosen Seifert surface and of the matrix representation of the Seifert form and that they give rise to a function
$$ \sigma_L: S^1 \setminus \{1\} \to \Z $$
which is piecewise constant.

\subsection{Multivariate Levine-Tristram signature}
\label{secondo}
Here we introduce the multivariate signature. Most of the content of this subsection can be found in \cite{CimasoniFlorens}.

There are some concepts related to knots that do not extend uniquely to links. For example we might want to endow our link with extra structure and give an ordering to the components: we obtain an ordered link and we say that two such links are isotopic if there is an isotopy that respects the ordering as well. 
Interpolating between links and ordered links there are coloured links:
a $\mu$-coloured link $L$ is an oriented link in $S^3$ together with a surjective map assigning to each component of $L$ a colour in $\{1, \ldots, \mu\}$. 

We will often write $L= L_1 \sqcup \cdots \sqcup L_\mu$ where $L_i$ is the sublink of the components of $L$ of colour $i$.

Notice that a $1$-coloured link is an ordinary link, while an ordered link $L$ is a $\mu$-coloured link if $\mu$ is equal to the number of connected components of $L$.

In order to define the multivariate signatures we need $C$-complexes, which are a generalization of Seifert surfaces for coloured links. They were first introduced by Cooper in \cite{Cooper} and studied by Kadokami \cite{Kadokami} and Cimasoni \cite{Cimasoni}.

A $C$-complex for a $\mu$-coloured link $L=L_1 \sqcup \cdots \sqcup L_\mu$ is a union $S= S_1 \cup \cdots \cup S_\mu$ of surfaces in $S^3$ such that:
\begin{itemize}
    \item for all $i$ the surface $S_i$ is a Seifert surface for $L_i$. We allow in this case $S_i$ to be disconnected, but there should be no closed components;
    \item for all $i \neq j$, $S_i$ and $S_j$ intersect transversely in a possibly empty union of \textit{clasp intersections} (see Figure \ref{clasp});
    \item there are no triple intersections: for $i,j,k$ pairwise distinct, $S_i \cap S_j \cap S_k$ is empty.
\end{itemize}

It is proved in \cite{Cimasoni} that every coloured link has a $C$-complex.

\begin{figure}[H]
    \centering
    \includegraphics[width = 5cm]{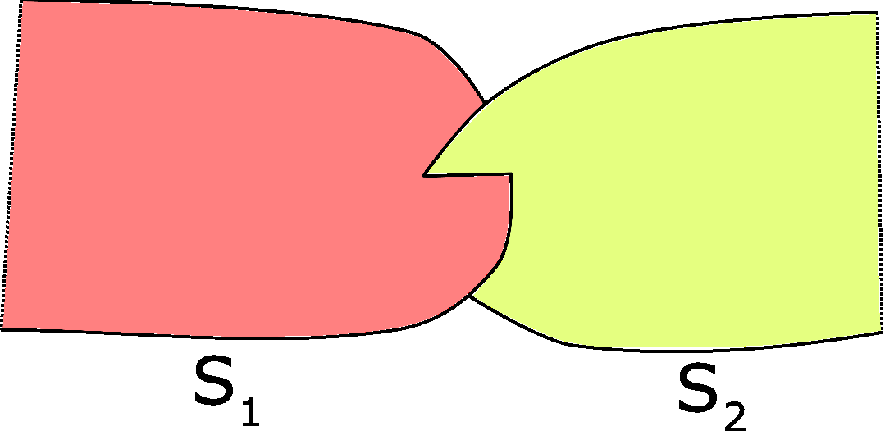}
    \caption{A clasp intersection.}
    \label{clasp}
\end{figure}

In order to define the ordinary Levine-Tristram signature we defined some ``push-off'' maps $i_\varepsilon$ on the first homology group of a Seifert surface for our link and then we used them to define the Seifert form. It turns out (see \cite{CimasoniFlorens}) that, given a $C$-complex $S= S_1 \cup \ldots \cup S_\mu$ there are similar maps 
$$i_\varepsilon: H_1(S) \to H_1(S^3 \setminus S) $$
where $\varepsilon \in \{\pm 1 \}^\mu$. These maps push off (the homology class of) a curve in $S$ positively on $S_i$ if $\varepsilon_i$ is positive and negatively otherwise.
Hence we can define
$$\alpha_\varepsilon: H_1(S) \times H_1(S) \to \Z $$
as $\alpha_\varepsilon(x,y)= \lk(i_\varepsilon(x),y).$

%Given a $C$-complex $S= S_1 \cup \ldots \cup S_\mu$ be a $C$-complex for a $\mu$-coloured link $L$. Let $N_i \cong S_i \times [-1,1]$ be a tubular neighbourhood of $S_i$. For $\varepsilon_i \in \{\pm 1\}$ let $S_i^{\varepsilon_i}$ denote the surface $S_i \times \{ \varepsilon_i\} \subset N_i$. 

%Let $Y$ be the complement of $\bigcup_{i=1}^\mu \operatorname{int} (N_i)$ in $S^3$. Then, for $\varepsilon=(\varepsilon_1,\ldots,\varepsilon_\mu) \in \{\pm 1 \}^\mu$, set

%$$S^\varepsilon = \bigcup_{i=1}^\mu S_i^{\varepsilon_i} \cap Y. $$
%\begin{figure}[H]
 %   \centering
 %   \includegraphics[width= 7cm]{Seifertform.eps}
%\end{figure}

%Since the intersections are clasps, there is an obvious homotopy equivalence between $S$ and $S^\varepsilon$, inducing an isomorphism $H_1(S) \to H_1(S^\varepsilon)$. Let
%$$i_\varepsilon: H_1(S) \to H_1(S^\varepsilon) \to H_1(S^3 \setminus S) $$
%be the composition of this isomorphism with the inclusion induced homomorphism. We are now ready to define
%$$\alpha_\varepsilon:H_1(S) \times H_1(S) \to \Z $$
%to be the bilinear form 
%$$ \alpha_\varepsilon(x,y)=\operatorname{lk}(i_\varepsilon(x), y). $$
Fix a basis of $H_1(S)$ and let $A_\varepsilon$ denote the matrix of $\alpha_\varepsilon$.
Notice that when $\mu=1$, $\alpha_{-1}$ is the usual Seifert form. 
Moreover, for all $\varepsilon$, the identity $A_{-\varepsilon}=A_{\varepsilon}^T$ holds.

Fix $\omega= (\omega_1,\ldots, \omega_\mu) \in T^\mu_*$, where $T^\mu _* = (S^1 \setminus \{1\})^\mu$. Set
$$H(\omega) \coloneqq \sum_{\varepsilon \in \{\pm 1\}^\mu} \prod_{i=1}^\mu (1 - \overline{\omega}_i^{\varepsilon_i}) A_\varepsilon. $$

Notice that the matrix $H(\omega)$ is Hermitian, because $\overline{H(\omega)}=H(\overline{\omega})= H(\omega)^T$.

\begin{dfn}
The Levine-Tristram multivariate signature $\sigma_\omega(L)$ of a $\mu$-coloured link $L$ at $\omega \in T^\mu_*$ is
$$\sigma_\omega(L) \coloneqq \sign (H(\omega)). $$
\end{dfn}

The following was proved by Cimasoni and Florens in \cite{CimasoniFlorens}.
\begin{thm}
The multivariate signature is a well-defined isotopy invariant of coloured links.
\end{thm}

\subsection{Twisted homology}
\label{twistedhomology}

\label{twisted}

Twisted homology is a well-known topic in algebraic topology and a reference can be found for example in \cite{DavisKirk} or in Chapter 5 of \cite{ConwayPhD}.

Let $X$ be a connected CW complex and let $Y \subset X$ be a possibly empty subcomplex of $X$. Let $p: \widetilde{X} \to X$ be the universal cover of $X$ and define $\widetilde{Y} := p^{-1}(Y)$.

The fundamental group $\pi_1(X)$ acts on the left on $\widetilde{X}$ and this action gives to $C_*(\widetilde{X}, \widetilde{Y}):= C_*^{CW}(\widetilde{X}, \widetilde{Y})$ a left $\Z[\pi_1(X)]$-module structure.

Let $R$ be a ring and let $M$ be a left $R$-module and a right $\Z[\pi_1(X)]$-module such that the left and the right action are compatible, that is to say $M$ is a $(R, \Z[\pi_1(X)])$-bimodule. 

\begin{dfn}
The chain complex $C_*(X,Y; M) := M \otimes_{\Z[\pi_1(X)]} C_*(\widetilde{X}, \widetilde{Y})$ of left $R$-modules is said to be the twisted chain complex with coefficients in $M$. The correspondent homology left $R$-modules $H_*(X,Y;M)$ are called twisted homology modules of $(X,Y)$ with coefficients in $M$.
\end{dfn}

\begin{ex}
\label{strutturaComega}
Let $\psi: \pi_1 (X) \to \Z^\mu= \langle t_1,\ldots, t_\mu \rangle$ be an homomorphism and fix $$\omega = (\omega_1, \ldots, \omega_\mu) \in T^\mu= (S^1 \setminus \{1\})^\mu.$$ The homomorphism $\psi$ can be extended to $\Psi: \Z[\pi_1 (X)] \to \Z[\Z^\mu] = \Z[ t_1^{\pm1}, \ldots, t_{\mu}^{\pm1}]$ and can be composed with the ring homomorphism $\Z[ t_1^{\pm1}, \ldots, t_{\mu}^{\pm1}] \to \C$ which evaluates $t=(t_1, \ldots, t_\mu)$ in $\omega.$ We therefore obtain a ring homomorphism
$$\phi: \Z[\pi_1(X)] \to \C $$
which induces on $\C$ a $(\C, \Z[\pi_1(X)])$-bimodule structure.

In this situation, to emphasize the choice of $\omega$ we write $\C^\omega$ instead of simply $\C$. Therefore the $(\C, \Z[\pi_1(X)]) $-bimodule structure on $\C^\omega$ allows us to consider the complex vector spaces $H_k(X; \C^\omega).$ When it is not clear from the context we will indicate $\C^\omega$ as $\C^{\psi,\omega}$ to remember the dependence on $\psi$ too.

Suppose now $X$ to be a manifold of dimension $2k$. Similarly to the untwisted case, on $H_k(X; \C^\omega)$ is defined a twisted intersection pairing:
$$\lambda_M(W): H_k(X;\C^\omega) \times H_k(X;\C^\omega) \to \C. $$
This pairing can be proved to be Hermitian when $k$ is even and it is skew-Hermitian when $k$ is odd. As in the untwisted case, the radical of this pairing corresponds to the image of the inclusion of $H_k(\partial X;\C^\omega). $
\end{ex}

\begin{rmk}
Let $L=L_1\cup \ldots \cup L_\mu$ be a link, and consider a collection $F=F_1 \cup \ldots \cup F_\mu$ of oriented and connected surfaces that are smoothly and properly embedded in the four ball $D^4$ in general position\footnote{i.e. their only intersections are transverse double points between different surfaces.} and such that $\partial F_i= L_i$. Let $W_F$ denote the exterior of $F$ in $D^4$. It can be seen (see Lemma \ref{finalmente}) that $H_1(W_F) \simeq \Z^\mu$ is generated by the meridians of the $F_i$'s.

Let $\omega= (\omega_1,\ldots, \omega_\mu) \in T^\mu_*$. Sending the meridians of $L_i$ to $\omega_i$ gives a homomorphism
$$\pi_1(W_F) \to H_1(W_F) \to \C $$
and therefore we obtain twisted homology complex vector spaces $$ H_k(W_F; \C^\omega).$$
As it was seen in Example \ref{strutturaComega}, on $H_2(W_F; \C^\omega)$ there is a Hermitian twisted intersection form $\lambda_{\C^\omega}(W_F)$. The following result is due to Conway, Nagel, Toffoli and its proof can be found in \cite{ConwayNagelToffoli}.

\begin{thm}
\label{CNT}
Let $L$ be a $\mu$-coloured link. For every coloured bounding surface F and for all $\omega \in T^\mu_*$, we have
$$\sigma_\omega(L) = \sign(\lambda_{\C^\omega}(W_F)). $$
\end{thm}

This algebraic $4$-dimensional interpretation of the multivariate signature extends the one by Cimasoni and Florens in \cite{CimasoniFlorens}, which makes use of the more geometric branched covers of the four ball, and which was used by Cimasoni and Conway to prove their theorem in \cite{CimasoniConway}. The idea is to use this more general interpretation to obtain an extension of their result. 
\end{rmk}

\subsection{Novikov-Wall non-additivity}
\label{sec-2}
In this section we state an important result by Wall \cite{Wall} about the non-additivity of (twisted) signatures under gluing of $4k$-manifolds along their boundary. This theorem, together with the $4$-dimensional interpretation of the multivariate signature in Theorem \ref{CNT}, will be of fundamental importance in the proof of our main result (Theorem \ref{main}).
Let $(H, \lambda)$ be a finite dimensional skew-Hermitian complex vector space and let $L_1$, $L_2$ and $L_3$ be three totally isotropic subspaces of $H$, i.e. such that $\lambda$ vanishes identically on $L_i$ for $i=1,2,3$. Consider the Hermitian form $f$ defined on $(L_1+L_2) \cap L_3$ as follows: for $a,b \in (L_1+L_2) \cap L_3$, write $a=a_1+a_2$ with $a_i \in L_i$ for $i=1,2$ and set
$$f(a,b)= \lambda(a_2,b).$$
The form $f$ can easily be checked to be well-defined and Hermitian. We can therefore define:

\begin{dfn}
The signature of $f$ is called the Maslov index of $L_1,L_2$ and $L_3$. It will be denoted by $\Maslov(L_1,L_2,L_3).$
\end{dfn}

\begin{rmk}
\label{propertymaslov}
It is not difficult to show that:
\begin{itemize}
    \item if $L_1,L_2,L_3$ are isotropic subspaces of $H$ and $L'_1,L'_2,L'_3$ are isotropic subspaces of $H'$, then $L_i \oplus L'_i$ for $i=1,2,3$ is an isotropic subspace of $H \oplus H'$ and 
    $$\Maslov(L_1\oplus L'_1,L_2\oplus L'_2, L_3\oplus L'_3)= \Maslov(L_1,L_2, L_3)+  \Maslov(L'_1,L'_2, L'_3);$$
    \item If $\sigma$ is a permutation of $(1,2,3)$ it holds:
$$\Maslov (L_1,L_2,L_3)= \sgn(\sigma) \Maslov(L_{\sigma(1)},L_{\sigma(2)},L_{\sigma(3)}). $$
\end{itemize}
\end{rmk}

Let $Y$ be an oriented connected compact $4k$-manifold and let $X_0$ be an oriented compact $(4k-1)$-manifold, properly embedded into $Y$ so that $\partial X_0= X_0 \cap \partial Y$. Suppose that $X_0$ splits $Y$ into two manifolds $Y_-$ and $Y_+$. For $\varepsilon= \pm$, denote by $X_\varepsilon$ the closure of $\partial Y_\varepsilon \setminus X_0$, which is a compact $(4k-1)$-manifold. Let $Z$ denote the compact $(4k-2)$-manifold $$Z= \partial X_0 = \partial X_+ = \partial X_-. $$
The manifolds $Y_+$ and $Y_-$ inherit an orientation from $Y$. Orient $X_0$, $X_+$ and $X_-$ so that 
$$ \partial Y_+ = X_+ \cup (-X_0)$$ and $$ \partial Y_- = X_0 \cup (-X_-)$$
and orient $Z$ such that $$Z= \partial X_- = \partial X_+ = \partial X_0.  $$

Suppose that we are given a homomorphism $$\psi: \pi_1(Y) \to \Z^\mu = \langle t_1,\ldots, t_\mu \rangle . $$
By Example \ref{strutturaComega}, for every choice of $\omega \in (S^1 \setminus \{1\})^\mu$, this endows $\C$ with a $(\C, \Z[\pi_1(Y)])$-bimodule structure, given by the map $\Psi: \Z[\pi_1(Y)] \to \C$, and we emphasize the choice of $\omega$ by indicating this bimodule as $\C^\omega$. For every submanifold $S$ of $Y$ the inclusion induced homomorphism $ \pi_1(S) \to \pi_1(Y)$ induces on $\C$ a $(\C, \Z[\pi_1(S)])$-bimodule structure, which we still indicate as $\C^\omega$. Recall that for manifolds $W$ of even dimension we also have a twisted intersection pairing $\lambda_{\C^\omega}(W)$ and that in the case the dimension is a multiple of $4$ the intersection pairing is Hermitian, and therefore it has a well-defined signature $\sign_\omega(W)$.

\begin{figure}[H]
    \centering
    \includegraphics[width = 5 cm]{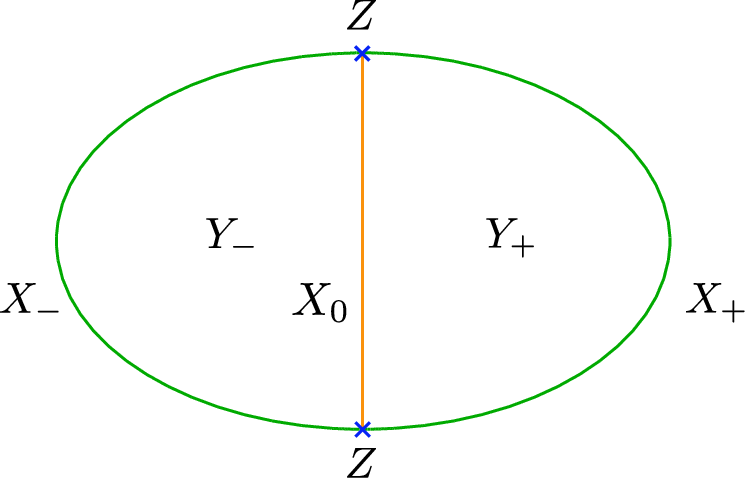}
    \caption{The manifold $Y$ in the setting of the Novikov-Wall non-additivity theorem.}
\end{figure}
\begin{thm}[Novikov-Wall non-additivity theorem]
In the situation above, 
$$\sign_\omega(Y)= \sign_\omega(Y_+) + \sign_\omega(Y_-) + \Maslov(L_-,L_0, L_+)  $$
where $L_\varepsilon= \ker (H_{2k-1}(Z; \C^\omega) \to H_{2k-1}(X_\varepsilon ; \C^\omega)) $ for $\varepsilon=-, +, 0$.
\end{thm}

\begin{rmk}
This theorem was originally proved by Wall in \cite{Wall} in the untwisted case, but it can be easily extended to the twisted case.
\end{rmk}

\section{The isotropic functor}
\label{section: isotropic functor}
The main goal of this section is to find a generalization of the reduced Burau representation, which is defined on braid groups to coloured tangles. We will see that coloured tangles are the morphisms of a category and we will therefore find a functor $\mathscr{F}_\omega$ on this category to a certain category of skew-Hermitian vector spaces. 

\label{sec-3}
\subsection{The category of coloured tangles}
Braids can be seen as a particular example of tangles. Intuitively, a tangle is a particular properly embedded $1$-dimensional submanifold of $D^2 \times [0,1]$ with a prescribed oriented boundary. More precisely, given a positive integer $n$, let $p_j^{(n)}$ be the point $$\left(\frac{2j-n-1}{n},0 \right)\in D^2 $$
for $j=1,\ldots, n$. A \emph{tangle} $\tau$ is a smooth, properly embedded $1$-submanifold of $D^2 \times [0,1]$ such that there exist $n$ and $n'$
 such that 
 $$\partial\tau=\{p_1^{(n)},\ldots,p_n^{(n)}\}\times \{0\} \cup \{p_1^{(n')},\ldots,p_{n'}^{(n')}\}\times \{1\} $$
 
 %Let $\varepsilon$ and $\varepsilon'$ be sequences of $\pm 1$'s of length $n$ and $n'$ respectively.

%\begin{dfn}
%An $(\varepsilon, \varepsilon')$-tangle is an %oriented, smooth, properly embedded %$1$-submanifold of $D^2 \times [0,1]$ whose %oriented boundary is $$\sum_{j=1}^{n'} %\varepsilon_j'(p^{n'}_j,1) -  \sum_{j=1}^{n} %\varepsilon_j(p^{n}_j,0).$$
%\end{dfn}

%\begin{figure}[H]
  %  \centering
 %   \includegraphics[width = 4 cm]{tanglebw.eps}
 %   \caption{A $(\varepsilon, %\varepsilon')$-tangle, where %$\varepsilon=(-1,+1,+1)$ and %$\varepsilon'=+1$}
%    \label{tangle}
%\end{figure}

An oriented tangle $\tau$ is called $\mu$-coloured if each of its components is assigned a label in $\{1,\ldots,\mu\}$. We will call a $\mu$-coloured oriented tangle a $(c,c')$-tangle, where 
$$c:\{1,\ldots,n\} \to \{\pm1,\ldots,\pm\mu\},$$  
$$c':\{1,\ldots,n'\} \to \{\pm1,\ldots,\pm\mu\}$$ 
are sequences so that $c(k)= +j$ or $-j$ if $p_k^{(n)}$ belongs to a component with colour $j$ and the sign depends only by the induced orientation on the point as boundary of an oriented string of $\tau$. We sometimes denote a colouring $c:\{1,\ldots,n\}\to \{\pm 1,\ldots, \pm \mu\}$ as $(c(1),\ldots,c(n))$.

Two $(c,c')$-tangles $\tau_1$ and $\tau_2$ are isotopic if there exists a self-homeomorphism $h$ of $D^2 \times [0,1]$ that keeps the boundary $\partial(D^2 \times[0,1])$ fixed and such that $h$ sends $\tau_1$ homeomorphically to $\tau_2$ preserving the orientation and the colouring.
We denote by $T_\mu(c,c')$ the set of isotopy classes of $(c,c')$-tangles. Let $\id_c$ indicate the isotopy class of the trivial $(c,c)$-tangle $\{p_1^{(n)}, \ldots,p_n^{(n)}\} \times [0,1] $. Given a $(c,c')$-tangle $\tau_1$ and a $(c',c'')$-tangle $\tau_2$ such that they both meet the boundary of $D^2 \times [0,1]$ orthogonally, their composition is the $(c,c'')$-tangle $\tau_2 \tau_1$ obtained by gluing the two $D^2\times [0,1]$ along the disk corresponding to $c'$ and then shrinking the height of the resulting $D^2 \times [0,2]$ by a factor $2$.
\begin{figure}[H]
    \centering
    \includegraphics[width = 7cm]{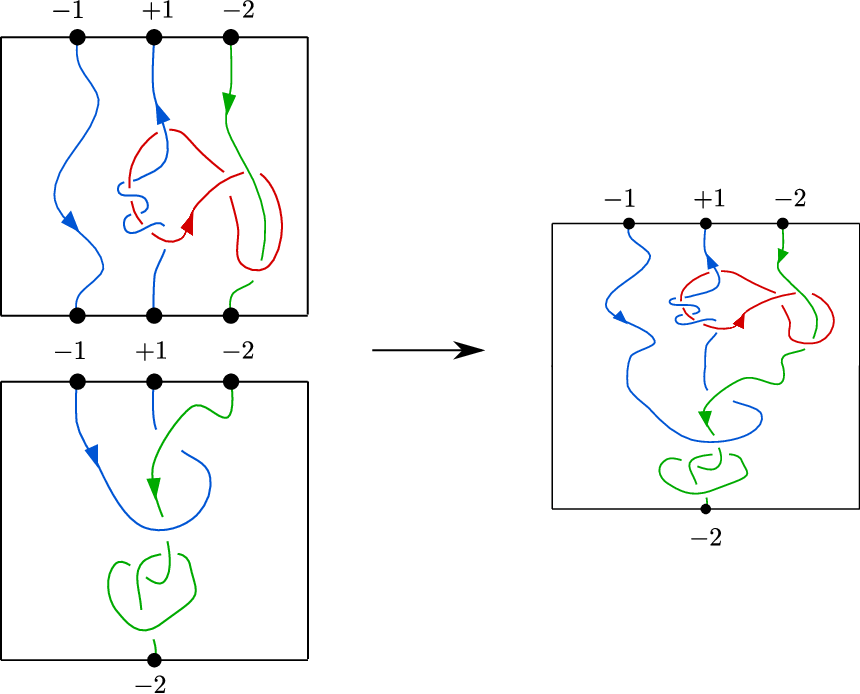}
    \caption{The composition of two coloured tangles. Here $\tau_1$ is a $(c,c')$-tangle with $c=(-2)$ and $c'=(-1,+1,-2)$, while $\tau_2$ is a $(c',c'')$-tangle with $c''=(-1,+1,-2)$.}
    \label{composizionetangle}
\end{figure}
Since in every isotopy class of tangles there is a representative that meets the boundary of $D^2 \times [0,1]$ orthogonally, the composition of tangles induces a composition
$$T_\mu(c,c') \times T_\mu(c',c'') \to T_\mu(c,c'') $$
on the isotopy classes of $\mu$-coloured tangles.

A $\mu$-coloured $(c,c)$-tangle $\tau \subset D^2 \times [0,1]$ is called a braid if the projection of each of the connected components of $\tau$ to $[0,1]$ is a homeomorphism. 

\begin{rmk}
All colourings $c:\{1,\ldots,n\} \to \{\pm 1,\pm 2, \ldots, \pm \mu\}$ as objects, and the isotopy classes of $(c,c')$-tangles as morphisms $c \to c'$ form a category $\Tangles_\mu$ called the category of $\mu$-coloured tangles.
\end{rmk}

\begin{dfn}
Given an endomorphism in $\Tangles_\mu$, that is to say $\tau \in T_\mu(c,c)$, one can define its closure as the $\mu$-coloured link $\widehat{\tau}\subset S^3$ obtained from $\tau$ adding $n$ oriented coloured parallel strands in $S^3 \setminus(D^2 \times [0,1])$ connecting $p_j^{(n)} \times \{1\}$ to $p_j^{(n)} \times \{0\}$ for $j=1,\ldots,n$. Here $n$ denotes the length of the sequence $c$.
\end{dfn}

\begin{figure}[H]
    \centering
    \includegraphics[width = 3 cm]{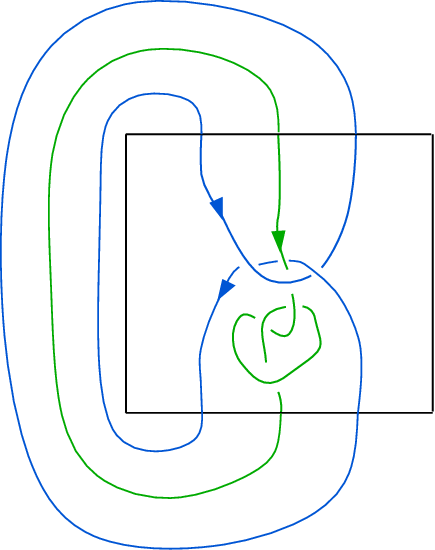}
    \caption{The closure of a tangle.}
\end{figure}

\subsection{The isotropic category}
We introduce the category $\Isotr_\Lambda$ of isotropic relations over a ring $\Lambda$. Fix an integral domain $\Lambda$ endowed with a ring involution $a \mapsto \overline{a}$.

A \emph{skew-Hermitian $\Lambda$-module} $H$ is a finitely generated $\Lambda$-module endowed with a possibly degenerate skew-Hermitian form $\lambda$.

Given a skew-Hermitian $\Lambda$-module $H$, $-H$ will indicate the same module endowed with the opposite form $-\lambda$.

Given a submodule $V$ of a skew-Hermitian $\Lambda$-module $H$, the \emph{annihilator} of $V$ is the submodule
$$ \Ann(V)= \{x \in H \, | \, \lambda(v,x)= 0 \text{ for all } v \in V \}. $$

A submodule $V$ of a skew-Hermitian $\Lambda$-module $H$ is called \emph{totally isotropic} if $V \subset \Ann(V)$ or, equivalently, if $\lambda$ vanishes identically on $V$.

%\begin{dfn}
%A submodule $V$ of a skew-Hermitian $\Lambda$-module $H$ is called Lagrangian if $V = \Ann(V)$.
%\end{dfn}

Given two skew-Hermitian $\Lambda$-modules $H_1$ and $H_2$, a totally isotropic submodule $N$ of $(-H_1) \oplus H_2$ is called an isotropic relation and we denote it by $N:H_1 \Longrightarrow H_2$. If $H$ is a skew-Hermitian $\Lambda$-module, an example of an isotropic relation $H \Longrightarrow H$ is the diagonal $\Delta_H=\{(h,h) \in (-H) \oplus H\}$.

If $N_1$ and $N_2$ are two isotropic relations $N_1:H_1 \Longrightarrow H_2$ and $N_2:H_2 \Longrightarrow H_3$ their composition
$$
    N_2 \circ N_1 :H_1 \Longrightarrow H_3 
$$

is defined as 
$$ 
N_2 \circ N_1 = \{(h_1,h_3) \, | \, \text{there exists } h_2 \in H_2 \text{ such that } (h_1,h_2) \in N_1 \text{ and } (h_2, h_3) \in N_2\}.
$$

Isotropic relations can be interpreted as a generalization of unitary isomorphisms: in fact if $\gamma: H_1 \to H_2$ is a unitary isomorphism, that is to say it preserves the skew-Hermitian forms, then its graph $\Gamma_\gamma\subset (-H_1) \oplus H_2$ is an isotropic submodule of $(-H_1)\oplus H_2$ and hence it is an isotropic relation $H_1 \Longrightarrow H_2$.
It is straightforward to see:
\begin{prop}
Hermitan $\Lambda$-modules as objects and isotropic relations as morphisms form a category $\Isotr_\Lambda$.
\end{prop}

\subsection{The isotropic functor}

Given a positive integer $n$, let $p_j^{(n)}$ indicate the point $\left(\frac{2j-n-1}{n},0\right)$ in the closed unitary disk $D^2$. Let $\mathcal{N}_n$ be an open tubular neighbourhood of the $n$-uple $\{p_1^{(n)},\ldots,p_n^{(n)}\}$. Let us denote by $D_n$ the space $D^2 \setminus \mathcal{N}_n$. We endow $D_n$ with the orientation inherited from $D^2$ and we fix a basepoint $z \in \partial D^2 \subset D_n$. 

\begin{figure}[H]
    \centering
    \includegraphics[width = 4cm]{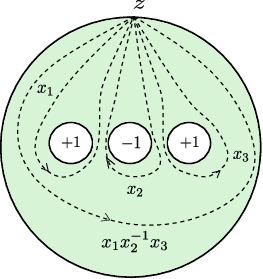}
    \caption{An example of $D_n$ for $\mu=1$.}
    \label{fig:my_label}
\end{figure}

Let $c:\{1,\ldots,n\} \to \{\pm 1, \ldots, \pm \mu\}$ be a map such that for all $i$, $1 \le i \le \mu$, either $+i$ or $-i$ is in the range of $c$. 

The fundamental group $\pi_1(D_n,z)$ is freely generated by $\{x_1^c,\ldots,x_n^c\}$ where $x_j^c$ indicates a simple loop which goes one time around $p_j$ counterclockwise if $\sgn(c(j))=+1$ and otherwise it goes clockwise. We hence obtain a map
\begin{align*}
\psi_c: \pi_1(D_n,z) &\to \langle t_1,\ldots, t_\mu \rangle \simeq \Z^\mu \\
x_j^c & \mapsto t_{|c(j)|}.
\end{align*}

Given $\omega \in T^\mu_*$ we can consider the twisted homology modules $H_*(D_n, \C^{\psi_c, \, \omega})$
and we are also given a skew-Hermitian intersection form $$\lambda_{\C^{\psi_c, \, \omega}}(D_n): H_1(D_n, \C^{\psi_c, \, \omega}) \times H_1(D_n, \C^{\psi_c, \, \omega}) \to \C$$ which we will denote $\lambda_{c, \, \omega}(D_n)$.

For any $\tau \in T_\mu(c,c')=\Hom_{\Tangles_\mu}(c,c')$ let $\mathcal{N}(\tau)$ denote an open tubular neighbourhood of $\tau$ and let $X_\tau = D^2 \times I \setminus \mathcal{N}(\tau)$ be the exterior of the tangle. 
By Mayer-Vietoris we get that $H_1(X_\tau)\simeq \bigoplus_{j=1}^m \Z m_j$ where $m_j$ is the meridian of the $j$-th connected component of $\tau.$ We obtain a map $$\pi_1(X_\tau) \to H_1(X_\tau) \to \langle t_1,\ldots,t_\mu \rangle =\Z^\mu$$ where the first arrow is the Hurewicz map and the second one assigns to $m_j$  the canonical generator $t_{i}$ such that $i$ is the colour of the $j$-th component of the tangle. 
We can then consider $H_*(X_\tau, \C^\omega).$ 

Let $i_c: H_1(D_n, \C^{\psi_c, \, \omega}) \to H_1(X_\tau, \C^\omega)$ and $i_{c'}: H_1(D_{n'}, \C^{\psi_{c'}, \, \omega}) \to H_1(X_\tau, \C^\omega)$ be the maps induced by the inclusion in $X_\tau$.

We denote by $$j_{\tau}:H_1(D_{n}, \C^{\psi_c, \, \omega}) \oplus H_1(D_{n'}, \C^{\psi_c, \, \omega}) \to H_1(X_\tau, \C^\omega)$$ the map given by $j_\tau(x,x')=i_{c'}(x')-i_c(x).$

\begin{thm}
\label{functor}
Let $\omega \in T^\mu_*$. Given $c:\{1,\ldots,n\} \to \{\pm 1, \ldots, \pm \mu \}$ define $$\mathscr{F}_\omega(c):=(H_1(D_n, \C^{\psi_c, \, \omega}), \langle \, , \rangle_c),$$ while for every tangle $\tau \in T_\mu(c,c')$ define $$\mathscr{F}_{\omega}(\tau):= \ker j_{\tau}.$$
Then $\mathcal{F}_\omega$ defines a functor $$\mathscr{F}_{\omega}:\Tangles_\mu \to \Isotr_{\mathbb{C}}.$$ 
%which makes the following diagram commute:
%\[
%\begin{tikzcd}
%\Braids_\mu \arrow[r] \arrow[d] & \Tangles_\mu \arrow[d, "\mathscr{F}_{\omega}"] \\
%\widetilde{U}_{\mathbb{C}} \arrow[r, "\Gamma"] & \Isotr_{\mathbb{C}}
%\end{tikzcd}\]
\end{thm}

\begin{proof}

The pair $(H_1(D_n; \C^{\psi_c, \, \omega}), \lambda_{c,\omega}(D_n))$ is clearly an object of the isotropic category.

Let us fix $\tau \in T_\mu(c,c')$ and let us show that $\F_\omega(\tau)=\ker j_{\tau}$ is a totally isotropic subspace of $(-H_1(D_n; \C^{\psi_c, \, \omega}))\oplus H_1(D_{n'}; \C^{\psi_{c'}, \, \omega})$. Denote by $\lambda$ the form $\lambda_{\psi_c, \, \omega}(D_n) $ and by $\lambda'$ the form $\lambda_{c', \, \omega}(D_{n'})$. \\ For any fixed $(x,x'), (y,y') \in \ker j_\tau$, consider $(-\lambda) \oplus \lambda'((x,x'),(y,y'))=-\lambda(x,y)+\lambda'(x',y').$

Let us denote by $\Sigma$ the boundary $\partial X_{\tau}$ of the exterior of the tangle and consider the following orientation preserving inclusions:
$$i:-D_n \to \Sigma$$
$$i':D_{n'} \to \Sigma. $$

The inclusion of $\Sigma$ in $X_\tau$ induces twisted homology modules $H_*(\Sigma;M)$ and a twisted intersection form $\lambda_{\C^\omega}(\Sigma)$.

Notice that the inclusion $-D_n \sqcup D_{n'} \to \Sigma$ is orientation preserving and the forms $-\lambda \oplus \lambda'$ and $ \lambda_{\C^\omega}(\Sigma)$ are compatible via the inclusion.

Let $\varphi:H_1(\Sigma; \C^\omega) \to H_1( X_\tau; \C^\omega)$ be the inclusion induced map and let $k$ be the map 
\begin{align*}
k:H_1(D_{n}; \C^{\psi_c, \, \omega}) \oplus H_1(D_{n'}; \C^{\psi_c, \, \omega}) & \to H_1(\Sigma; \C^\omega) \\
(x,x') &\mapsto i'_*(x')-i_*(x).
\end{align*}
Then we have $$\ker j_\tau = \ker(\varphi \circ k)$$ %with
%\begin{align*}
%    \psi: H_1(D_{n}; \C^{\psi_c, \, \omega}) \oplus H_1(D_{n'}; \C^{\psi_c, \, \omega}) &\to H_1(D_{n}; \C^{\psi_c, \, \omega}) \oplus H_1(D_{n'}; \C^{\psi_c, \, \omega}) \\
%    (x,x') & \mapsto (-x,x')
%\end{align*}

%Therefore $\Ann(\ker(j_\tau))= \Ann(\psi(\ker(\varphi \circ k)))= \psi (\Ann(\ker(\varphi \circ k)))$.

We shall now check that $\ker(\varphi \circ k)$ is totally isotropic. \\
Fix $(x,x'), (y,y') \in \ker(\varphi \circ k)$. We have that $ k(x,x')$ and  $k(y,y')$ belong to $\ker \varphi$. It is a well-known consequence of Poincaré duality that, given a compact $(2k+1)$-manifold $X$ with boundary, the kernel in $H_k(\partial X; \C^\omega)$ of the map induced by the inclusion of the boundary is totally isotropic (actually Lagrangian, see Subsection \ref{finito}). Hence in our case

$$ \lambda_{\C^\omega}(\Sigma) (k(x,x'), k(y,y'))=0$$
and then clearly $-\lambda(x,y)+\lambda'(x',y')=0$. This shows that $\ker(j_\tau)$ is totally isotropic.

%We have shown
%$$ \ker (\varphi \circ k) \subset \Ann (\ker (\varphi \circ k))$$
%and therefore
%$$\ker j_\tau= \psi (\ker (\varphi \circ k)) \subset \psi(\Ann (\ker (\varphi \circ k)))= \Ann (\ker j_\tau).$$

The proof of the functoriality follows from an argument similar to the one in \cite{CimasoniTuraev}, Lemma 3.4.

\end{proof}

\section{Non-additivity of multivariate signature}
\label{sec-4}
We are finally ready to state and prove the extension of the result by Cimasoni and Conway \cite{CimasoniConway} and we devolve the first part of this section to this scope. The second part focuses on studying the case of coloured braids and we see that in this situation the formula of Cimasoni and Conway is exactly the multivariable version of the Gambaudo and Ghys formula.

For $\tau$ a $(c,c')$-tangle, let $\overline{\tau}$ denote the $(c',c)$-tangle obtained by reflecting $\tau$ along a horizontal plane and inverting the orientation.
For $j \in \{1,\ldots \mu\}$ let $A_j$ be the set of $k \in \{1,\ldots, n\}$ such that $c(k)= \pm j$. Define 
$$i_j = \sum_{k \in A_j} \sgn c(k). $$
\begin{thm}
\label{main}
Let $c$ be an object in $\Tangles_\mu$. For any $(c,c)$-tangles $\tau_1, \tau_2$
$$\sigma_\omega(\widehat{\tau_1\tau_2}) -\sigma_\omega(\widehat{\tau}_1)-\sigma_\omega(\widehat{\tau}_2) = \Maslov(\F_\omega(\overline{\tau}_1), \Delta, \F_\omega(\tau_2)) $$
for all $\omega\in T^\mu_*$ such that $\prod_{j=1}^\mu \omega_j^{i_j} \neq 1.$
\end{thm}

Here $\Delta$ indicates the diagonal of $H_1(D_n; \C^{\psi_c,\omega}).$

The strategy will be to employ the Novikov-Wall non-additivity theorem and a $4$-dimensional interpretation of the multivariate signature, studied by Conway, Nagel and Toffoli in \cite{ConwayNagelToffoli} and stated in Theorem \ref{CNT}, in order to construct a $4$-manifold which has twisted signature 
$$\sigma_\omega(\widehat{\tau_1\tau_2}) -\sigma_\omega(\widehat{\tau}_1)-\sigma_\omega(\widehat{\tau}_2) - \Maslov(\F_\omega(\overline{\tau}_1), \Delta, \F_\omega(\tau_2)) $$
and then prove that this manifold must have zero signature.

\subsection{Proof of the main theorem}

We start by proving an important lemma which explains the restriction on the $\omega$'s in Theorem \ref{main}. We will see in Subsection \ref{finito} that this restriction is necessary for the equality in Theorem \ref{main} to hold.
\begin{lem}
\label{nondeg}
Let $c$ be an object of $\Tangles_\mu$. Let $l(c)=(i_1,\ldots,i_\mu)\in\Z^\mu$, where $$i_j= \sum_{k\in A_j} \sgn c(k) .$$
Then for every $\omega =(\omega_1,\ldots, \omega_\mu) \in T^\mu_*$ such that 
$$ \prod_{j=1}^\mu \omega_j^{i_j} \neq 1 $$
the form $\lambda_{\C^{\psi_c, \, \omega}}(D_n)$ is non-degenerate. 
\end{lem}

\begin{proof}
As noted in subsection \ref{twistedhomology}, the radical of this form is the image of the inclusion induced map $H_1(\partial D_n; \C^{\psi_c, \, \omega}) \to H_1(D_n; \C^{\psi_c, \, \omega})$. 

Since $\partial D_n= S_1 \sqcup \ldots \sqcup S_{n+1}$ is disconnected, 
$$ H_1(\partial D_n; \C^{\psi_c, \, \omega}) = \bigoplus _{j=1}^{n+1} H_1 (S_j; \C^\omega).$$

Assume that for $j \le n$ the component $S_j$ is the one corresponding to the $j$-th puncture $p_j$ and that $S_{n+1}$ is $\partial D^2$.

We can see $S_j$ as a CW complex with one $0$-cell $P_j$ and one $1$-cell $e_j$.
For every $j$ the free $\Z[\pi_1(S_j)]$-modules $C_k(\widetilde{S_j})$ are generated by a lift of the $k$-cells in $S_j$, hence $C_0(\widetilde{S_j}) \simeq \Z[\pi_1(S_j)]$ is generated by a lift $\widetilde{P}_j$ of the $0$-cell and $C_1(\widetilde{S_j}) \simeq \Z[\pi_1(S_j)]$ generated by a lift $\widetilde{e}_j$ of the $1$-cell which starts at $\tilde{P}_j$, while $C_k(\widetilde{S}_j)=0$ for $k>1$. 

Recall that $\psi_c: \pi_1(D_n) \to \C$ was defined by sending the generator $x_i$ to $\omega_i$, where $x_i$ is counterclockwise oriented if $\sgn c(i) =+1$ and clockwise oriented otherwise.
The $\Z[\pi_1(D_n)]$-module $\C^{\psi_c, \, \omega}$ can be seen as a $\Z[\pi_1(S_j)]$-module by restricting the action of $\Z[\pi_1(D_n)]$:
the inclusions (and the ``obvious'' choice of path connecting the basepoint $z=P_{n+1}$ for $D_n$ with the basepoint for $S_j$) induce maps
\begin{align*}
\pi_1(S_j, P_j) &\to \pi_1(D_n,z) \\
t_j &\to x_j^{\sgn c(j)}
\end{align*}
for all $j\le n$, where $t_j$ indicates the generator of $\pi_1(S_j)$ counterclockwise oriented. 

For $j=n+1$ the map is
\begin{align*}
\pi_1(S_{n+1}, z) &\to \pi_1(D_n,z) \\
t_{n+1} &\to \prod_{j=1}^n x_j^{\sgn c(j)}
\end{align*}
where $t_{n+1}$ indicates the generator of $\pi_1(S_{n+1})$ counterclockwise oriented. 

We write $ \Z[\pi_1(S_j)] = \Z[t_j^{\pm 1}]$. The chain complex $\C^\omega \otimes_{\Z[t_j^{\pm 1}]} C_*(\widetilde{S}_j)$ is
$$ 0 \longrightarrow \C \xrightarrow{\id \otimes \partial} \C \longrightarrow 0 $$
since $\C^\omega \otimes_{\Z[t_j^{\pm 1}]}\Z[t_j^{\pm 1}] = \C $ as complex vector spaces. 

Since $\partial (\widetilde{e}_j)= t_j \widetilde{P}_j-\widetilde{P}_j$, the map $\id \otimes \partial$ is multiplication by $\omega_j^{\sgn c(j)}-1$ if $j \le n$ and by $\left(\prod_{k=1}^\mu \omega_k^{i_k}\right)-1$ if $j=n+1$. Therefore, when $ \prod_{k=1}^\mu \omega_k^{i_k} \neq 1 $ the map
$\id \otimes \partial$ is an isomorphism for all $S_j$ and $H_1(S_j; \C^\omega)=0$ for $1 \le j \le n+1$.

Hence in our setting $H_1(\partial D_n; \C^{\psi_c, \, \omega})=\{0\}$ and the radical of the twisted form on $D_n$ is trivial.

\end{proof}

\begin{rmk}
Given an object in $\Tangles_\mu$ and $\omega \in T^\mu_*$ we will denote from now on by $I_c(\omega)$ the complex number $\prod_{j=1}^\mu \omega_j^{i_j}$, where $$i_j= \sum_{k \in A_j} \sgn c(k).$$

For every $\omega \in T^\mu_*$ we will denote the full subcategory of $\Tangles_\mu$ with objects the set of $c:\{1,\ldots, n\} \to \{\pm 1, \ldots , \pm \mu\}$ such that $I_c(\omega) \neq 1$ by $\Tangles^\omega_\mu.$
\end{rmk}

Given two objects $c, c'$ in $\Tangles_\mu$, by $c \sqcup c'$ we will mean the object of $\Tangles_\mu$ obtained by juxtaposition of $c$ and $c'$. Similarly given two tangles $\tau, \tau'$, by $\tau \sqcup \tau'$ we will mean the tangle obtained by juxtaposition of the two tangles. 

\begin{prop}
\label{ort}
Let $\tau_1$ be a $(c_1,c_1')$-tangle and $\tau_2$ a $(c_2,c_2')$-tangle. 
Let $\omega \in T^\mu_*$ be such that $\tau_1$ and $\tau_2$ are morphisms in the category $\Tangles_\mu^\omega$. Then we have

\begin{itemize}
    \item $H_1(D_{n_1 + n_2};\C^{\psi_{c_1\sqcup c_2}, \omega}) \simeq H_1(D_{n_1};\C^{\psi_{c_1}, \omega}) \oplus H_1(D_{ n_2};\C^{\psi_{ c_2}, \omega}) \oplus \C$ and also \newline $H_1(D_{n'_1 + n'_2};\C^{\psi_{c'_1\sqcup c'_2}, \omega}) \simeq H_1(D_{n'_1};\C^{\psi_{c'_1}, \omega}) \oplus H_1(D_{ n'_2};\C^{\psi_{c'_2}, \omega}) \oplus \C$ as complex vector spaces endowed with a skew-Hermitian form;
    \item $\F_\omega(\tau_1 \sqcup \tau_2) \simeq \F_\omega(\tau_1) \oplus \F_\omega(\tau_2) \oplus \Delta_{\C}$, where $ \Delta_{\C}$ is the diagonal of $\C$;
\end{itemize}

\end{prop}

\begin{proof}
The space $D_{n_1 + n_2}$ can be obtained by gluing $D_{n_1}$ and $D_{n_2}$ along an interval in their boundary, which is contractible. Similarly, $X_{\tau_1 \sqcup \tau_2}$ can be obtained by gluing $X_{\tau_1}$ and $X_{\tau_2}$ along a square in their boundary, which is contractible as well. For $n=n_1,n_2$, since $D_n$ is homotopically equivalent to $\bigvee_n S^1$ a quick Mayer-Vietoris argument, together with the calculation of the twisted homology groups of $S^1$ similarly to the proof of Lemma \ref{nondeg}, shows that the $0$-th twisted homology module of $D_n$ is trivial. 

Then the Mayer-Vietoris sequence gives 

\[
\begin{tikzcd}[column sep=small]
0 \arrow[r] & H_1(D_{n_1};\C^{\psi_{c_1}, \omega}) \oplus H_1(D_{ n_2};\C^{\psi_{c_2}, \omega}) \arrow[r, "f"] \arrow[d, "i_{c_1} \oplus i_{c_2}"] &  H_1(D_{n_1 + n_2};\C^{\psi_{c_1\sqcup c_2}, \omega}) \arrow[r, "\partial_1"] \arrow[d, "i_{c_1 \sqcup c_2}"] & \C \arrow[d, equal] \arrow[r] & 0\\
0 \arrow[r] & H_1(X_{\tau_1};\C^\omega) \oplus H_1(X_{ \tau_2};\C^\omega) \arrow[r]  &  H_1(X_{\tau_1 \sqcup \tau_2};\C^\omega) \arrow[r, "\partial_2"] & \C &
\end{tikzcd}
\]
Hence $\partial_2$ is surjective as well. Notice that $f$ actually preserves the intersection forms.

We can split these exact sequences of vector spaces and obtain
$$ H_1(D_{n_1 + n_2};\C^{\psi_{c_1\sqcup c_2}, \omega}) \simeq  H_1(D_{n_1};\C^{\psi_{c_1}, \omega}) \oplus H_1(D_{ n_2};\C^{\psi_{ c_2}, \omega}) \oplus \C $$
and 

$$H_1(X_{\tau_1 \sqcup \tau_2};\C^\omega) \simeq  H_1(X_{\tau_1};\C^\omega) \oplus H_1(X_{ \tau_2};\C^\omega) \oplus \C.$$
With these identifications the map $i_{c_1 \sqcup c_2}$ becomes $i_{c_1} \oplus i_{c_2} \oplus \id_\C$.
Then $\F_\omega(\tau_1 \sqcup \tau_2)=\ker(j_{\tau_1 \sqcup \tau_2})= \ker(j_{\tau_1 }) \oplus \ker(j_{\tau_2 }) \oplus \Delta_{\C}$ as claimed.

We still need to check that the first decomposition is orthogonal. Since both $\lambda_{c_1,\omega}(D_{n_1})$ and $\lambda_{c_2,\omega}(D_{n_2})$ are non-degenerate because of the previous lemma, then we can choose the decomposition to be orthogonal. 

\end{proof}

\subsubsection{A reduction}

Let $A$ and $B$ be two maps which associate to each tangle an integer. We will write
$$ A \simeq B $$
if $|A-B|$ is bounded by a constant which is independent of the tangles.

We shall show that in order to prove Theorem \ref{main} we only need to prove the looser statement

$$\sigma_{\omega}( \widehat{\tau_1 \tau _2}) - \sigma_{\omega}( \widehat{\tau}_1) -\sigma_{\omega}( \widehat{\tau }_2) \simeq \Maslov(\F_{\omega}(\overline{\tau}_1), \Delta, \F_\omega(\tau_2)).$$

This reduction will allow us to not keep track of most of the Novikov-Wall defects.

To do this, however, we first need to prove one preliminary result.
We will indicate for simplicity by $ M_\omega(\tau_1, \tau_2)$ the right-hand side of the equation.

Recall that the left-hand side of this equation is additive with respect to the disjoint union of tangles. 

Let $\tau_1$ and $\tau_2$ be $(c,c)$-tangles and $\tau_1', \tau_2'$ be $(c',c')$-tangles. Let $\omega \in T^\mu_*$ be such that both $c$ and $c'$ belong to $\Tangles_\mu^\omega$. Then Proposition \ref{ort}, with $c_1=c_1'=c$ and $c_2=c_2'=c' $, and the invariance under unitary isomorphisms of the Maslov index % da cambiare \ref{Maslovindex} 
imply
$$M_\omega(\tau_1 \sqcup \tau'_1, \tau_2 \sqcup \tau'_2)= \Maslov(\F_\omega(\overline{\tau}_1) \oplus \F_\omega(\overline{\tau'}_1) \oplus \Delta_{\C}, \Delta, \F_\omega(\tau_2) \oplus \F_\omega(\tau_2') \oplus \Delta_{\C}). $$
Notice that in this case the diagonal $\Delta$ is equal to $\Delta_{H_1(D_n; \C^{\psi_c,\omega})} \oplus \Delta_{H_1(D_{n'}; \C^{\psi_{c^\prime},\omega})} \oplus \Delta_\C. $ 

The additivity property of the Maslov index in Remark \ref{propertymaslov}
immediately imply that
$$M_\omega(\tau_1 \sqcup \tau'_1, \tau_2 \sqcup \tau'_2) = M_\omega(\tau_1, \tau_2) + M_\omega(\tau'_1, \tau'_2)+ \Maslov(\Delta_\C,\Delta_\C,\Delta_\C),$$
but $\Maslov(\Delta_\C,\Delta_\C,\Delta_\C) = 0. $

However, even if $\tau_1, \tau_2$ are both morphisms in $T_\mu^\omega(c,c)$, $\tau_1 \sqcup \tau_2$ does not need to be a morphism in $\Tangles_\mu^\omega$, and this explains the technicalities in the proof of the following lemma.

\begin{lem}
\label{add}
Let $\omega \in T^\mu_*$ and let $c$ be an object in $\Tangles_\mu^\omega$. For any $\tau_1, \tau_2$ $(c,c)$-tangles there exists $c'$ in $\Tangles^\omega_\mu$ and two $(c',c')$-tangles $\tau'_1$ and $\tau'_2$ such that
$$ M_\omega(\tau'_1, \tau'_2)= 2 M_\omega(\tau_1,\tau_2).$$
\end{lem}

\begin{proof}
If $I_c(\omega) \neq -1$ then $\tau'_1= \tau_1 \sqcup \tau_1$ and $\tau'_2= \tau_2 \sqcup \tau_2$ both work, that is to say $I_{c\sqcup c}(\omega) \neq 1$ and $M_\omega(\tau'_1,\tau'_2)= 2 M_\omega(\tau_1,\tau_2)$ as shown above.

Otherwise we consider 
$$ \tau'_1=\tau_1 \sqcup \tau_1 \sqcup l $$
and 
$$ \tau'_2=\tau_2 \sqcup \tau_2 \sqcup l $$
where $l$ is a trivial vertical upward oriented arc which we colour with the first $j\in \{1,\ldots,\mu\}$ such that $\omega_j \neq -1$ if $\omega \neq (-1,\ldots, -1)$ and we colour it with $j=1$ otherwise. If $c'$ denotes the corresponding colouring, we see that $I_{c'}(\omega)= \omega_j \neq 1.$ 

If we can choose $\omega_j \neq -1$, thanks to Proposition \ref{ort} we see that 

$$M_\omega(\tau'_1, \tau'_2) = M_\omega(\tau_1 \sqcup \tau_1 \sqcup l, \tau_2  \sqcup \tau_2 \sqcup l) $$

factors as $$M_\omega(\tau_1, \tau_2)+ M_\omega(\tau_1 \sqcup l, \tau_2 \sqcup l)= 2M_\omega(\tau_1, \tau_2) + M_\omega(l,l) $$

but $M_\omega(l,l) $ vanishes thanks to the properties of the Maslov index in Remark \ref{propertymaslov}.
This is true since $\tau_i, \tau_i \sqcup l$ and $l$ are all morphisms in $\Tangles^\omega_\mu$ for $i=1,2$.

If $\omega= (-1,\ldots,-1)$ we need to be more careful.
Notice that
$$ M_\omega(\tau'_1, \tau'_2) = M_\omega(\tau_1 \sqcup \tau_1 \sqcup l \sqcup l, \tau_2  \sqcup \tau_2 \sqcup l \sqcup l)$$
thanks to Remark \ref{propertymaslov} and Proposition \ref{ort}. Here we are using that $\tau_1 \sqcup \tau_1 \sqcup l$ and $l$ are both morphisms in $\Tangles^\omega_\mu$. 

Therefore this factors as 
$$M_\omega(\tau_1, \tau_2)+  M_\omega(\tau_1 \sqcup l \sqcup l, \tau_2  \sqcup l \sqcup l)$$

and the second term is equal to 

$$M_\omega(\tau_1 \sqcup l \sqcup l \sqcup l, \tau_2  \sqcup l \sqcup l \sqcup l)= M_\omega(\tau_1, \tau_2) + M_\omega( l \sqcup l \sqcup l,  l \sqcup l \sqcup l) $$

but we know that $M_\omega( l \sqcup l \sqcup l,  l \sqcup l \sqcup l)=0 $, hence we get that also in this case 

$$M_\omega(\tau'_1, \tau'_2)= 2 M_\omega(\tau_1,\tau_2). $$

\end{proof}

We are now ready to prove the reduction:

\begin{prop}
\label{reduction}
To prove Theorem \ref{main} it is enough to show that for any $c$ in $\Tangles_\mu$ we have
$$\sigma_{\omega}( \widehat{\tau_1 \tau _2}) - \sigma_{\omega}( \widehat{\tau}_1) -\sigma_{\omega}( \widehat{\tau}_2) \simeq \Maslov(\F_{\omega}(\overline{\tau}_1), \Delta, \F_\omega(\tau_2))$$
for each $\omega$ such that $I_c(\omega) \neq 1$.
\end{prop}

\begin{proof}
Assume by contradiction that for a fixed object $c$ in $\Tangles_\mu$ there are two $(c,c)$-tangles $\tau_1,\tau_2$ and an $\omega$ with $I_c(\omega) \neq 1$ such that

$$N_\omega(\tau_1,\tau_2):= \sigma_{\omega}( \widehat{\tau_1 \tau _2}) - \sigma_{\omega}( \widehat{\tau}_1) -\sigma_{\omega}( \widehat{\tau} _2) - \Maslov(\F_{\omega}(\overline{\tau}_1), \Delta, \F_\omega(\tau_2)) $$
does not vanish. Then for any $m \in \N$ we can use inductively Lemma \ref{add} to obtain a colouring $c(m)$ and tangles $\tau_1(m)$ and $\tau_2(m)$ such that $I_{c(m)}(\omega) \neq 1$ and 

$$M_\omega(\tau_1(m),\tau_2(m))=2^m M_\omega(\tau_1,\tau_2). $$

Moreover, as a consequence of the proof of Lemma \ref{add}, we can assume that $\tau_i(m)$ is the juxtaposition of $m$ copies of $\tau_i$ and of some trivial vertical arcs. Therefore, the additivity by juxtaposition of the Levine-Tristram signature implies 
$$\sigma_{\omega}( \widehat{\tau_1(m) \tau _2(m)}) - \sigma_{\omega}( \widehat{\tau}_1(m)) -\sigma_{\omega}( \widehat{\tau} _2(m))= 2^m (\sigma_{\omega}( \widehat{\tau_1 \tau _2}) - \sigma_{\omega}( \widehat{\tau}_1) -\sigma_{\omega}( \widehat{\tau} _2)). $$
Hence
$$ N_\omega(\tau_1(m),\tau_2(m))=2^m N_\omega(\tau_1,\tau_2). $$ 

Since the right-hand side goes to infinity as $m$ grows, this concludes the proof.

\end{proof}

\subsubsection{The manifold $P(\tau_1, \tau_2)$.}

We now construct a $4$-manifold which has twisted signature equal to $M_\omega(\tau_1, \tau_2) $ up to a uniformly bounded constant.

Fix a map $c:\{1,\ldots,n\} \to \{ \pm 1, \ldots, \pm \mu\}$ and an element $\omega\in T^\mu_*$ such that $I_c(\omega) \neq 1$. 

Let $P$ be the pair of pants with a fixed orientation as in the figure below. Let $I_1$ and $I_2$ be closed intervals joining the inner boundary components of the pair of pants to the outer boundary component. Thicken these intervals in order to get $J_1=[0,1]\times I_1$ and $J_2= [0,1] \times I_2$ as shown in Figure \ref{pantalone}.

\begin{figure}[H]
    \centering
    \includegraphics[width= 6 cm]{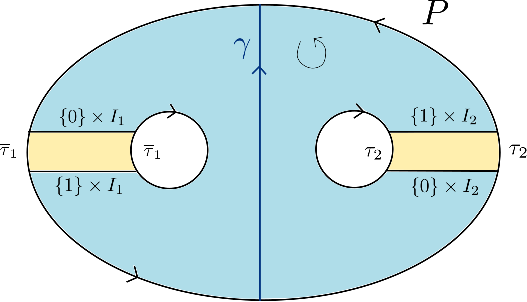}
    \caption{The pair of pants $P$. Notice that the tangle $\tau_1$ is reflected.}
    \label{pantalone}
\end{figure}

Let $\tau_1,\tau_2$ be $(c,c)$-tangles. Consider the surface $T(\tau_1,\tau_2)$ in $ D^2 \times P$ defined by the following:

\begin{itemize}
    \item $T(\tau_1,\tau_2) \cap (D^2 \times (P \setminus (J_1 \cup J_2))) $ is the surface $ \{p_1^{(n)},\ldots, p_n^{(n)}\} \times (P \setminus (J_1 \cup J_2)) $;
    \item $T(\tau_1,\tau_2) \cap (D^2\times J_1)=T(\tau_1,\tau_2) \cap (D^2 \times [0,1] \times I_1) $ is the surface $\tau_1 \times I_1$;
    \item $T(\tau_1,\tau_2) \cap (D^2\times J_2)=T(\tau_1,\tau_2) \cap (D^2 \times [0,1] \times I_2) $ is the surface $\tau_2 \times I_2$.
\end{itemize}

Let $R(\tau_1, \tau_2)$ indicate an open tubular neighbourhood of this surface.

We define the compact, oriented $4$-manifold

$$ P(\tau_1, \tau_2):= (D^2 \times P) \setminus R(\tau_1, \tau_2). $$

%Notice that for each point $p \in I_i$, $i=1,2$, the surface $T(\tau_1, \tau_2)$ contains a copy of the tangle $\tau_i$, hence $P(\tau_1, \tau_2)$ contains a copy of $X_{\tau_i}$ for each $p \in I_i$.

\begin{lem}
\label{P}
The first homology group of $ P(\tau_1, \tau_2)$ can be identified with $\Z^k \oplus \Z^2$, where $k$ is the number of connected components of $T(\tau_1,\tau_2)$.

There exists a homomorphism $$H_1(P(\tau_1, \tau_2)) \to \Z^\mu$$ which vanishes on the summand $\Z^2$ and whose composition with the homomorphism induced by the inclusion of any copy of $X_{\tau_i}$ into $P(\tau_1,\tau_2)$ coincides with the natural map $H_1(X_{\tau_i}) \to \Z^\mu$.
\end{lem}

\begin{proof}
The Mayer-Vietoris sequence applied to the pair $P(\tau_1, \tau_2)\cup R(\tau_1, \tau_2)= D^2 \times P$ easily gives us that 
$$H_1(P(\tau_1, \tau_2)) = \left(\bigoplus_i \Z \gamma_i \right)\oplus \Z^2$$
where $\gamma_i$ indicates the oriented meridian of a connected component of the surface $T(\tau_1, \tau_2)$ and the summand $\Z^2$ is given by the holes in $D^2 \times P$. Notice that the colouring of the tangles extends to a well-defined colouring of the surface $ T(\tau_1, \tau_2)$.

We can then define a homomorphism

\begin{equation*}
    H_1(P(\tau_1,\tau_2)) \to \Z^\mu \simeq \langle t_1,\ldots, t_\mu \rangle
\end{equation*}
    which assigns to the oriented meridian $\gamma_i$ associated to a $j$-coloured component of the surface the element $t_j$ of $\Z^\mu$ and sends the summand $\Z^2$ to $0$.
    
    The homomorphism induced by the inclusion $H_1 (X_{\tau_i}) \to H_1(P(\tau_1,\tau_2))$ can be easily seen to send a meridian of a $j$-coloured component of the tangle into a meridian of a connected component of the surface, which by definition must have the same colouring $j$. This concludes our proof.

\end{proof}

Let $\omega\in T^\mu_*$. 
Using the homomorphism of Lemma \ref{P} we can give $ \C$ a \newline $\Z[\pi_1(P(\tau_1,\tau_2))]$-module structure and talk about the twisted homology modules $$H_*(P(\tau_1,\tau_2); \C^\omega). $$

If we call $X_{\widehat{\tau}}$ the exterior of the closure of a tangle in $D^2 \times S^1$, notice that the boundary of $\overline{R(\tau_1,\tau_2) }$ consists of a part contained in the boundary of $D^2 \times P $, and of a part which is properly embedded in $D^2 \times P $, which we call $U(\tau_1,\tau_2)$. Notice that 
$$\partial P(\tau_1,\tau_2)= 
((-X_{\widehat{\tau}_1}) \sqcup (-X_{\widehat{\tau}_2}) \sqcup X_{\widehat{\tau_1 \tau_2}}) \cup \partial D^2 \times P \cup U(\tau_1,\tau_2) $$
%Call $\widetilde{P}(\tau_1,\tau_2)$ the manifold 
%$$ \widetilde{P}(\tau_1,\tau_2) = P(\tau_1,\tau_2) \cup_{\partial (D^2) \times P} (D^2 \times P).$$
%The Mayer-Vietoris sequence gives us that $H_1(\widetilde{P}(\tau_1,\tau_2)) \simeq \Z^k$ is generated by the meridians $\gamma_i$, hence we can naturally define a homomorphism
%$$ H_1(\widetilde{P}(\tau_1,\tau_2)) \to \Z^\mu $$
%such that the composition with the homomorphism induced by the inclusion of $P(\tau_1,\tau_2)$ in $ \widetilde{P}(\tau_1,\tau_2)$ gives us the map defined in Lemma \ref{P}. We can therefore consider the twisted homology modules of $\widetilde{P}(\tau_1,\tau_2)$. 

The following result will be used throughout this section:

\begin{lem}
\label{zero}
For every $(c,c)$-tangle $\tau$, if $\Sigma$ denotes the surface $\partial X_{\widehat{\tau}} \setminus (\partial D^2 \times S^1) = \partial \mathcal{N}(\widehat{\tau})$, then $H_*(\Sigma; \C^\omega)=0.$ 
\end{lem}

\begin{proof}
Let $K_i$ for $1 \le i \le m$ be a connected component of $\widehat{\tau}$ such that $\widehat{\tau}= K_1 \sqcup \ldots \sqcup K_m$. The surface $\Sigma$ is the union of $m$ disjoint tori $T_i$ and each $T_i$ inherits a well-defined colouring. 

For each $i$
\begin{equation*}
     \pi_1(T_i) \to \pi_1(X_{\widehat{\tau}})  \to H_1( X_{\widehat{\tau}}) \to \Z^\mu\simeq \langle t_1 , \ldots t_\mu \rangle
\end{equation*}
sends the meridian $m_i$ to $t_i$ and the longitude $l_i$ to some $f(t_1,\ldots,t_\mu) \in \langle t_1^{\pm1},\ldots,t_\mu^{\pm1}\rangle$.

Decompose the torus $T_i$ as a CW complex with one $0$-cell, two $1$-cells which are $m_i$ and $l_i$ and one $2$-cell. The complex $\C^\omega \otimes_{\Z[\pi_1(T_i)]} C_*(\widetilde{T}_i)$ is of free $\Z[\pi_1(T_i)]$-modules generated by lifts of the cells:
$$ 0 \to \C \xrightarrow{\phi} \C^2 \xrightarrow{\psi} \C \to 0 $$

Let $\widetilde{e}_i$ denote a lift of the $2$-cell, $\widetilde{m_i}, \widetilde{l_i}$ lifts of the meridian and the longitude respectively, and $\widetilde{p}_i$ a lift of the $0$-cell. Indeed $$\partial \widetilde{e}_i= \widetilde{m_i}+(m_i \cdot \widetilde{l_i})-(l_i \cdot \widetilde{m_i}) -\widetilde{l_i} $$ hence $\phi(z)=\left(z(1-f(\omega_1,\ldots,\omega_\mu)),z \cdot (\omega_i-1)\right)$ is injective.
Furthermore,
$$ \partial \widetilde{m_i}= m_i \cdot \widetilde{p}_i -\widetilde{p}_i$$
and 
$$ \partial \widetilde{l_i}= l_i \cdot \widetilde{p}_i -\widetilde{p}_i,$$
therefore $\psi(z,w)= z \cdot (\omega_i-1) + w \cdot (f(\omega_1,\ldots,\omega_\mu)-1)$, and $\ker \psi = \Imm \phi$.

This implies $$H_*(T_i; \C^\omega) =0.$$
\end{proof}

We are now ready to prove:

\begin{prop}
\label{PMaslov}
For any $\omega$ and $\tau_1,\tau_2$ as above, we have
$$ \sign_\omega( P(\tau_1,\tau_2)) \simeq  \Maslov(\F_\omega(\overline{\tau}_1), \Delta, \F_\omega(\tau_2)).$$
\end{prop}

Recall that $\F_\omega(\overline{\tau}_1), \Delta, \F_\omega(\tau_2) $ are totally isotropic subspaces of $$(-H_1(D_n, \C^{\psi_c, \omega}))\oplus H_1(D_n, \C^{\psi_c, \omega})$$ and that the symbol $\simeq$ means that the two quantities in the equation are equal up to a uniformly bounded constant.

\begin{proof}
%Let us apply the Novikov-Wall theorem first to $$ \widetilde{P}(\tau_1,\tau_2)= P(\tau_1,\tau_2) \cup_{\partial D^2 \times P} (D^2 \times P).$$ The space $D^2 \times P $ is homotopically equivalent to a 1 dimensional CW complex, hence it has no degree 2 homology and its signature vanishes. The surface $Z= \partial D^2 \times \partial P$ consists of the union of three disjoint tori. Therefore
%$$\sign_\omega(\widetilde{P}(\tau_1,\tau_2)) = \sign_\omega(P(\tau_1,\tau_2)) + \Maslov(H_1,H_0,H_2) $$
%where $H_i$ is a subspace of $H_1(Z;\C^\omega)$. Hence $ \Maslov(H_1,H_0,H_2)$ is bounded by the dimension of $H_1(Z;\C^\omega)$. Since $Z$ does not depend on $\tau_1$ and $\tau_2$, we get that the difference between the twisted signatures of $P(\tau_1, \tau_2)$ and $ \widetilde{P}(\tau_1, \tau_2)$ is uniformly bounded.

%Let now $Y_\tau \supset X_\tau$ denote the exterior of a tangle $\tau$ in $S^2 \times [0,1] \supset D^2 \times [0,1]$.
%Moreover let $Y_{\widehat{\tau}} \supset X_{\widehat{\tau}}$ denote the exterior of the closure of a tangle $\tau$ in $S^2 \times S^1 \supset D^2 \times S^1$. 
Let $\gamma$ be a curve which divides the space $P$ in two cylinders (see Figure \ref{pantalone}). 
Call the manifold $(D^2 \times \gamma) \cap P(\tau_1,\tau_2)=X_{\id_c} $ as $X_0$. We cut $ P(\tau_1,\tau_2)$ along $X_0$. The two manifolds obtained are $M_1 \cong  X_{\widehat{\tau}_1} \times I_1$ and $M_2 \cong  X_{\widehat{\tau}_2} \times I_2 $.

Their boundaries split as $\partial M_i = X_{0}\cup_{\partial X_0} X_i$, where $$X_i = X_{\widehat{\tau}_i}\cup_{\partial X_{\widehat{\tau_i}}\times\{0\}}  \left(\partial X_{\widehat{\tau}_i} \times [0,1]\right) \cup_{\partial_v X_{\tau_i}\times\{1\}} X_{\tau_i},$$ 
where $\partial_v X_{\tau_i}$ indicates the vertical part of $\partial X_{\tau_i} \subset D^2 \times[0,1],$ i.e. $\partial X_{\tau_i}$ minus the top and bottom holed disks $D_n$. Notice that $\partial_v X_{\tau_i}$ is homotopically equivalent to a disjoint union of $S^1$'s (for the strings of $\tau$) and tori (for the closed components of $\tau$) and by a combination of Lemma \ref{zero} and the proof of Lemma \ref{nondeg} its twisted homology coincides with the one of $\partial D^2 \subset D_n$ (and therefore it vanishes if and only if $I_c(\omega) \neq 1$). 

Notice that $ \partial X_0$ is diffeomorphic to the double of $D_n$ and we call it $Z$.

The Novikov-Wall non-additivity theorem implies 
\begin{equation*}
    \sign_{\omega}({P}(\tau_1,\tau_2))= \sign_\omega(M_1)+\sign_\omega(M_2)+ \Maslov(L_1,L_0,L_2)
\end{equation*}

where $L_i= \ker (f_i:H_1(Z;\C^\omega) \to H_1(X_i; \C^\omega))$. 
The signatures of $M_1$ and $M_2$ vanish because they deformation retract onto a $3$-manifold.

As it is showed in the proof of Lemma \ref{nondeg} $H_1(\partial D_n; \C^\omega) = H_1(\partial D^2; \C^\omega)$ and  $ H_0(\partial D_n; \C^\omega)=H_0(\partial D^2; \C^\omega)$ (both twisted homology modules have dimension at most $1$ and they vanish if and only if $I_c(\omega) \neq 1$). The Mayer-Vietoris sequence gives $$\ldots \to H_1(\partial D^2; \C^\omega)\to H_1(D_n;\C^\omega) \oplus H_1(D_n;\C^\omega)\to H_1(Z;\C^\omega) \to H_0(\partial D^2; \C^\omega) \to 0,$$ 
hence we can fix a decomposition
$$ H_1(Z;\C^\omega) \cong (H_1(D_n;\C^\omega) \oplus H_1(D_n;\C^\omega))/\Delta_{\text{rad}} \oplus H_0(\partial D^2;\C^\omega). $$
Here $\Delta_\text{rad}$ indicates the diagonal of the radical of the form on $H_1(D_n;\C^\omega)) $, which coincides with the image of $H_1(\partial D^2;\C^\omega). $
The form $-\lambda_{\C^\omega}\oplus\lambda_{\C^\omega} $ induces a well-defined form on $(H_1(D_n;\C^\omega) \oplus H_1(D_n;\C^\omega))/\Delta_{\text{rad}}$ since $\Delta_{\text{rad}}$ is in the radical.
Call $\overline{f}_i$ the map $f_i \circ \iota$, where $$\iota:(H_1(D_n;\C^\omega) \oplus H_1(D_n;\C^\omega))/\Delta_{\text{rad}} \to H_1(Z;\C^\omega)$$ sends $(x,y)$ to $(x,y,0).$ Notice that $\iota$ preserves the intersection form. 
Hence $\iota(\ker \overline{f}_i) \subset L_i$ and $\dim \iota(\ker \overline{f}_i)$ differs from $\dim L_i $ of at most $1$. As a consequence, $\dim((\iota(\ker \overline{f}_1) + \iota(\ker \overline{f}_0)) \cap \iota(\ker \overline{f}_2))$ differs from $\dim(L_1+L_0)\cap L_2$ of at most $3$. Since $\iota$ is injective, this implies that
$$|\Maslov(L_1,L_0,L_2)-\Maslov(\ker \overline{f}_1,\ker \overline{f}_0,\ker \overline{f}_2)|\le 3. $$ 

The diagram of inclusion induced maps
\[
\begin{tikzcd}
H_1(D_n;\C^\omega) \arrow[r]\arrow[dr] & H_1(X_{\tau_i};\C^\omega ) \arrow[d,"g_i"]\\ & H_1(X_i; \C^\omega)
\end{tikzcd}
\]
commutes. Notice that we call the vertical map $g_i$.

The space $\partial_v X_{{\tau}_i} $ does not have necessarily vanishing twisted homology, but by Mayer-Vietoris the kernel of $H_1(X_{\tau_i};\C^\omega) \to H_1(X_i; \C^\omega) $ has at most dimension $1$, for $i=1,2$. 

Let us call
$$j_{0}: H_1(D_n;\C^{\psi_c,\omega}) \oplus H_1(D_n;\C^{\psi_c,\omega}) \to H_1(X_{\id_c}; \C^\omega) $$ the sum of the maps induced by the inclusion of $D_n \times \{0\}$ and $D_n \times \{1\}$ in $X_{\id_c}$. 

Similarly, let
$$j_{1}: H_1(D_n;\C^{\psi_c,\omega}) \oplus H_1(D_n;\C^{\psi_c,\omega}) \to H_1(X_{\overline{\tau}_1}; \C^\omega) $$
be the sum of the maps induced by the inclusion. Notice that we take the reflection of $\tau_1$.

Finally, let
$$j_{2}: H_1(D_n;\C^\omega) \oplus H_1(D_n;\C^\omega) \to H_1(X_{\tau_2}; \C^\omega) $$
be the sum of the maps induced by the inclusion too.

Then $\ker j_i \subset \ker (g_i \circ j_i) $ and their dimension differ at most by $1$. 

As a consequence 
$$| \Maslov(\ker j_1,\ker j_0,\ker j_2) - \Maslov(\ker(g_1 \circ j_1),\ker(g_0 \circ j_0),\ker (g_2 \circ j_2))| \le 3 $$
and therefore, since $\Delta_\text{rad}$ is contained in the radical of the form on $(H_1(D_n;\C^\omega))^2 $, 
$$\Maslov(\ker \overline{f}_1,\ker \overline{f}_0,\ker \overline{f}_2)\simeq  \Maslov(\ker j_1,\ker j_0,\ker j_2) $$

We have shown that 
$$ \sign_{\omega}({P}(\tau_1,\tau_2)) \simeq \Maslov(\ker(j_1),\ker(j_0),\ker(j_2)). $$

Recall that $\F_\omega (\id_c)= \ker( j_0 \circ \psi)$
 where $\psi = \left(-\id_{H_1(D_n;\C^{\psi_c, \, \omega})}\right) \oplus \id_{H_1(D_n;\C^{\psi_c, \, \omega})}$. As a consequence
 $$ \Maslov( \psi^{-1}(\ker(j_1)),\psi^{-1}(\ker(j_0)),\psi^{-1}(\ker(j_2)))=$$ $$ = \Maslov(\F_\omega(\overline{\tau}_1), \F_\omega(\id_c), \F_\omega(\tau_2)) $$
 which concludes the proof.

\end{proof}

\subsubsection{The manifold $C(\tau)$.}

In this section we will build a manifold that has the same twisted signature as the tangle closure, up to a uniformly bounded constant. 

Let $D^4$ denote the oriented unit $4$-ball and denote by $S^3$ its oriented boundary. Let $T=D^2 \times S^1 \hookrightarrow S^3$ be the standard embedding of the solid torus in $S^3$. 
Closing a coloured $(c,c)$-tangle $\tau \subset D^2 \times [0,1]$ yields a coloured link $\widehat{\tau}\subset T$.

Consider a collection $S(\tau)=F_1 \cup \ldots \cup F_\mu$ of oriented and connected surfaces smoothly and properly embedded in $D^4$ that are in general position (their only intersections are transverse double points between different surfaces) and such that for all $i$ the boundary $\partial F_i = F_i \cap S^3$ is the sublink of $\widehat{\tau}$ of colour $i$. 

It is not difficult to show that we can also assume that $S(\tau) \cap \frac{1}{2}S^3$ is the closure of the trivial tangle $\id_c$.

Moreover we will assume that the intersection of $S(\tau)$ with the closure of $D \setminus \frac{1}{2} D^4$ is contained in 
$$ N = \left\{x \in D^4 \,\Big{|}\, \frac{1}{2} \le \|x \| \le 1, \, \frac{x}{\|x\|} \in T \right\}\cong T \times [0,1].$$

One can easily check that such a collection of surfaces exists, and can be obtained for example by pushing a $C$-complex for $\widehat{\tau}$ inside $D^4$ in an appropriate way.

Let $\mathcal{N}(S(\tau)) = \mathcal{N}(F_1) \cup \ldots \cup \mathcal{N}(F_\mu)$ indicate the union of tubular neighbourhoods for the surfaces $F_i$ such that three different tubular neighbourhoods have empty intersection. Notice that $\partial \overline{\mathcal{N}(S(\tau))}$ consists of a part contained in $S^3$, and another part which is properly embedded in $D^4$ and which we call $U(\tau)$.
We will indicate as $C(\tau)$ the $4$-manifold
$$C(\tau)\coloneqq N \setminus (\mathcal{N}(S(\tau)) \cap N). $$
The boundary of $C(\tau)$ is 
$$ \partial C(\tau) = (U(\tau) \cap N) \cup_{\partial \mathcal{N}(\widehat{\tau}) \sqcup \partial \mathcal{N}(\widehat{\id}_c)} \left(X_{\widehat{\tau}} \sqcup (-X_{\widehat{\id}_c})\right) \cup_{\partial T \times \{0,1\}} \partial D^2 \times S^1 \times [0,1].$$

The proof of the following can be found in \cite{ConwayPhD}, Chapter 3.
\begin{lem}
\label{finalmente}
Let $F=F_1 \cup \ldots \cup F_\mu $ be the union of $\mu$ properly embedded, compact, connected and oriented surfaces $F_i \subset D^4$ which only intersect each other transversely in double points. Then $H_1(D^4 \setminus \mathcal{N}(F))$ is freely generated by the meridians of the components $F_i$, where $\mathcal{N}(F)$ indicates the union of open tubular neighbourhoods for the surfaces $F_i$.
\end{lem}

\begin{rmk}
Notice that the isomorphism in Lemma \ref{finalmente} induces a canonical map $$H_1(D^4 \setminus \mathcal{N}(S(\tau))) \to \Z^\mu. $$ 
Precomposing with the inclusion induced map $H_1(C(\tau)) \to H_1(D^4 \setminus \mathcal{N}(S(\tau)))$ and fixing $\omega \in T^\mu_*$ we get as usual a $\Z[\pi_1(C(\tau))]$-module structure on $\C$. 
\end{rmk}

\begin{thm}
\label{Ctau}
For every $\omega \in T^\mu_*$ and for every $(c,c)$-tangle $\tau$

$$ \sign_\omega (C(\tau))\simeq \sigma_\omega(\widehat{\tau}).$$
\end{thm}

\begin{proof}
Let $W_{\widehat{\tau}}$ indicate $D^4 \setminus \mathcal{N}(S(\tau))$. We know by the four-dimensional interpretation of the multivariate Levine-Tristram signature, studied by Conway, Nagel and Toffoli in \cite{ConwayNagelToffoli}, that
$$ \sign_\omega(W_{\widehat{\tau}}) = \sigma_\omega(\widehat{\tau}), $$
where $W_{\widehat{\tau}}$ is the exterior of a coloured bounding surface for $\widehat{\tau}$ in $D^4$.
We just need to show that 
$$ \sign_\omega(W_{\widehat{\tau}}) \simeq \sign_\omega (C(\tau)).$$ 

We can glue to $N \cong T \times [0,1]$ a copy of $\partial D^2 \times D^2 \times [0,1]$ along $\partial D^2 \times S^1 \times [0,1]$ to obtain the closed space $D^4 \setminus \frac{1}{2}\mathring{D}^4 \simeq S^3 \times [0,1]$. Hence gluing $C(\tau)$ and  $\partial D^2 \times D^2 \times [0,1]$ along $X_0:=\partial D^2 \times S^1 \times [0,1]$ gives us $M:=W_{\widehat{\tau}} \setminus (\frac{1}{2}\mathring{D}^4 \cap W_{\widehat{\tau}}) $.
Note that $\Sigma:= \partial X_0$ consists of two disjoint tori. The boundary of $C(\tau)$ is equal to  $(U(\tau) \cap N) \cup (X_{\widehat{\tau}} \sqcup (X_{\widehat{\id}_c})) \cup_{\Sigma} X_0$ while the boundary of $\partial D^2 \times D^2 \times [0,1]$ is $X_0 \cup_{\Sigma} (\partial D^2 \times D^2 \times \{0,1\})$. 

Since the Maslov index term is uniformly bounded by the dimension of $H_1(\Sigma; \C^\omega)$, Novikov-Wall yields

$$ \sign_\omega(M) \simeq \sign_\omega(C(\tau)) + \sign_\omega(\partial D^2 \times D^2 \times [0,1])= \sign_\omega(C(\tau)) $$
where the second term vanishes due to the fact that $\partial D^2 \times D^2 \times [0,1] $ deformation retracts onto $\partial D^2 .$

Next, glue $M$ and $W_{\widehat{\tau}} \cap \frac{1}{2}D^4 \cong W_{\widehat{\id}_c}$ along $Y_0:=\frac{1}{2}S^3 \setminus \mathcal{N}(\widehat{\id}_c)$.  
Lemma \ref{zero} applied to $\widehat{\id}_c$ implies that in this case Novikov-Wall theorem applies trivially:
$$ \sign_{\omega}(W_{\widehat{\tau}})= \sign_{\omega}(M)+ \sign_{\omega}(W_{\widehat{\id}_c})= \sign_{\omega}(M)+ \sigma_\omega(\widehat{\id}_c)= \sign_{\omega}(M).$$

This concludes the proof.

\end{proof}

\subsubsection{The manifold $M(\tau_1,\tau_2)$.}

Our goal now is to glue some copies of $C(\tau)$ and $P(\tau_1,\tau_2)$ along their boundary in order to obtain an oriented $4$-manifold $M(\tau_1,\tau_2)$.

Recall that 
$$\partial C(\tau)= (U(\tau) \cap N) \cup_{\partial \mathcal{N}(\widehat{\tau}) \sqcup \partial \mathcal{N}(\widehat{\id}_c)} \left(X_{\widehat{\tau}} \sqcup (-X_{\widehat{\id}_c})\right) \cup_{\partial T \times \{0,1\}} \partial D^2 \times S^1 \times [0,1] $$
and 
$$\partial P(\tau_1,\tau_2) = U(\tau_1,\tau_2) \cup_{V}((-X_{\widehat{\tau}_1}) \sqcup (-X_{\widehat{\tau}_2}) \sqcup X_{\widehat{\tau_1\tau_2}}) \cup_{\partial D^2 \times \partial P} (\partial D^2\times P),  $$
where $P$ denotes the pair of pants and $V= \partial \mathcal{N}(\widehat{\tau}_1) \sqcup \partial \mathcal{N}(\widehat{\tau}_2) \sqcup \partial \mathcal{N}(\widehat{\tau_1\tau_2})$. 

We can hence think of gluing $ P(\tau_1,\tau_2)$ on ``one side'' of the disjoint union of $C(\tau_1), C(\tau_2)$ and $C(\tau_1\tau_2)$ and $P(\id_c,\id_c)$ on ``the other side''. More precisely, set $M(\tau_1,\tau_2)$ equal to
$$(-P(\tau_1,\tau_2)) \cup_{Z}  \left(C(\tau_1) \sqcup C(\tau_2) \sqcup (-C(\tau_1\tau_2))\right) \cup_{W} P(\id_c,\id_c)$$
with $Z=X_{\widehat{\tau}_1} \sqcup X_{\widehat{\tau}_2} \sqcup X_{\widehat{\tau_1\tau_2}}$ and $W= X_{\widehat{\id}_c} \sqcup X_{\widehat{\id}_c} \sqcup X_{\widehat{\id}_c}.$

\begin{figure}[H]
    \centering
    \includegraphics[width = 9cm]{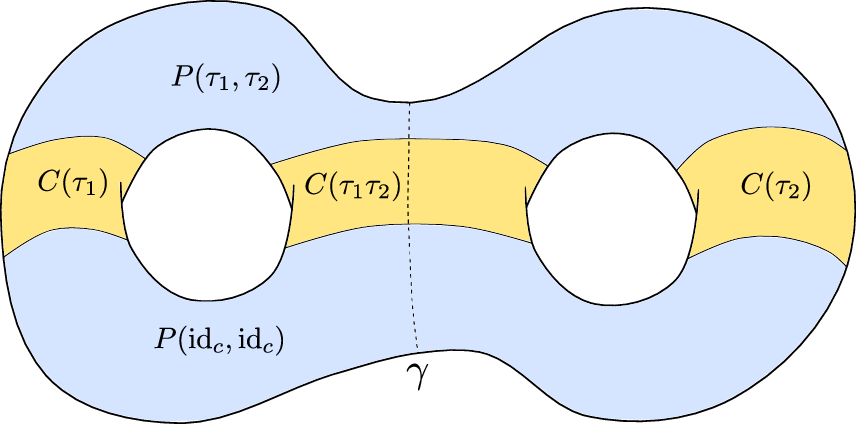}
    \caption{$M(\tau_1,\tau_2)$}
    \label{M}
\end{figure}

\begin{prop}
The $4$-manifold $M(\tau_1,\tau_2)$ can be endowed with an orientation such that $$\sign_\omega(M(\tau_1,\tau_2)) \simeq \sigma_{\omega}(\widehat{\tau_1\tau_2}) -  \sigma_{\omega}(\widehat{\tau}_1)- \sigma_{\omega}(\widehat{\tau}_2) -\Maslov(\F_\omega(\overline{\tau}_1), \F_\omega(\id_c), \F_\omega(\tau_2)).$$ 
\end{prop}

\begin{proof}
Start gluing $ -P(\tau_1,\tau_2)$ and $(-C(\tau_1)) \sqcup (-C(\tau_2)) \sqcup C(\tau_1\tau_2)$ along $$X_0= X_{\widehat{\tau}_1} \sqcup X_{\widehat{\tau}_2} \sqcup X_{\widehat{\tau_1\tau_2}}$$ and call $M$ the resulting manifold. 

Notice that $\partial X_0$ is the disjoint union of three tori $\partial D^2 \times S^1$ and of the boundaries of tubular neighbourhoods of the three tangles' closures: $$\partial \mathcal{N}(\widehat{\tau}_1), \, \partial \mathcal{N}(\widehat{\tau}_2), \, \partial \mathcal{N}(\widehat{\tau_1\tau_2}).$$ By Lemma \ref{zero} the twisted homology modules of the last three spaces vanish, hence the twisted homology of $ \partial X_0$ is the direct sum of the twisted homology of three tori and therefore the Maslov index term is bounded by the dimension of the first twisted homology module of $\partial X_0$, which does not depend on the tangles. Hence Proposition \ref{PMaslov} and Theorem \ref{Ctau} imply
$$ \sign_\omega (M)\simeq \sigma_{\omega}(\widehat{\tau_1\tau_2}) -  \sigma_{\omega}(\widehat{\tau}_1)- \sigma_{\omega}(\widehat{\tau}_2) -\Maslov(\F_\omega(\overline{\tau}_1), \F_\omega(\id_c), \F_\omega(\tau_2)).$$
We can repeat the same reasoning to the gluing of $P(\id_c, \id_c)$ and obtain that $\sign_\omega(M(\tau_1,\tau_2))$ equals $$\sigma_{\omega}(\widehat{\tau_1\tau_2}) -  \sigma_{\omega}(\widehat{\tau}_1)- \sigma_{\omega}(\widehat{\tau}_2) -\Maslov(\F_\omega(\overline{\tau}_1), \F_\omega(\id_c), \F_\omega(\tau_2))+  \Maslov (\Delta, \Delta, \Delta). $$
But $\Maslov (\Delta, \Delta, \Delta)=0 $, hence the result.

\end{proof}

To prove Theorem \ref{main} it only remains to show that the twisted signature of $M(\tau_1,\tau_2)$ is uniformly bounded by a constant.

Notice that $M(\tau_1, \tau_2)$ is contained in the manifold
$ D^2 \times S_2  $
where $S_2$ indicates the genus $2$ surface obtained by gluing to the pair of pants three cylinders $S^1 \times [0,1]$ and then gluing another pair of pants.

\begin{lem}
For a good choice of surfaces $S(\tau_1), S(\tau_2)$ and $S(\tau_1\tau_2)$ in the construction of $C(\tau_1),$ $C(\tau_2)$ and $C(\tau_1\tau_2)$ there exists a curve $\gamma$ in $S_2$ such that $ D^2 \times \gamma $ intersects the surface $T(\tau_1,\tau_2) \cup S(\tau_1)\cup S(\tau_2) \cup S(\tau_1\tau_2) \cup T(\id_c,\id_c) \subset D^2 \times S_2$ in $n$ disjoint circles $\{x_1,\ldots, x_n\}\times \gamma$.
\end{lem}

We will use the curve $\gamma$ to cut the manifold $M(\tau_1,\tau_2)$ in two pieces and apply Novikov-Wall.

\begin{proof}
The strategy is to build the curve $\gamma$ from four intervals $\gamma_1, \gamma_2,\gamma_3, \gamma_4$  alternatively contained in the pair of pants and in the central cylinder as in Figure \ref{M}. 
To do this we have to choose $S(\tau_1\tau_2)$ carefully. 

Let $F_1$ and $F_2$ indicate $C$-complexes for the coloured links $\widehat{\tau}_1$ and $\widehat{\tau}_2$, and construct a complex for $\widehat{\tau_1\tau_2}$ by connecting $F_1$ and $F_2$ along disjoint bands far from the tangles. 

\begin{figure}[H]
    \centering
    \includegraphics[width = 8 cm]{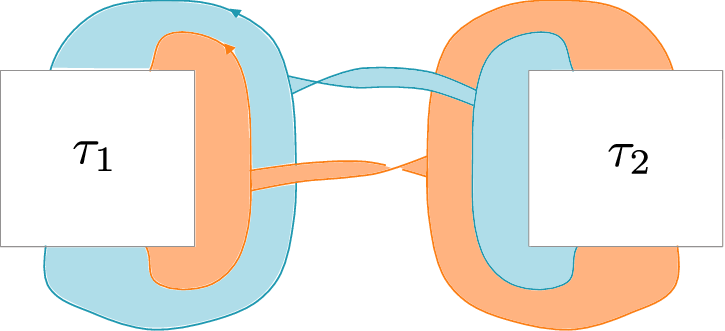}
    \caption{Construct a $C$-complex for $\widehat{\tau_1\tau_2}$ adding to $F_1$ and $F_2$ disjoint bands far from the tangles. }
\end{figure}
We can move this $C$-complex by an isotopy so that its boundary lies in $D^2 \times S^1 \subset S^3 = \partial D^4$ and then push its interior inside $D^4$ to obtain $S(\tau_1\tau_2)$ in such a way that there are two points $s_1, s_2 \in S^1$, such that the disks $D^2 \times \{s_i\} \subset D^2 \times S^1$ are far from the tangles and separate $\tau_1$ from $\tau_2$, and such that $S(\tau_1\tau_2)$ intersects $D^2 \times \{s_i\} \times [0,1]\subset N$ along $\{x_1,\ldots, x_n\} \times \{s_i\} \times [0,1]. $ 

Define $\gamma_1$ to be a curve in $P$ as in Figure \ref{pantalone} with endpoints $s_1$ and $s_2$, $\gamma_2$ to be the curve $\{s_1\} \times [0,1]$ in the cylinder $S^1 \times[0,1]$, $\gamma_3$ to be as $\gamma_1$ but in the other pair of pants and $\gamma_4$ to be the curve $\{s_2\} \times [0,1]$ in the same cylinder as $\gamma_2$. Gluing these intervals produces the required curve $\gamma$.

\begin{figure}[H]
    \centering
    \includegraphics[width = 4 cm]{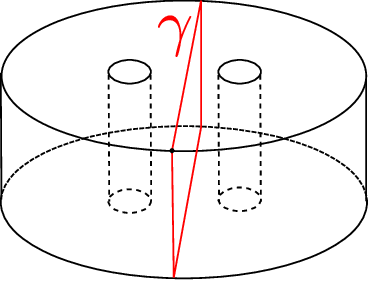}
    \caption{The curve $\gamma$ is contained in $\partial (P \times [0,1]) \cong S_2$.}
\end{figure}

\end{proof}

\begin{prop} For all $\omega \in T^\mu_*$ we have
$$\sign_{\omega}(M(\tau_1,\tau_2)) \simeq 0.$$
\end{prop}

\begin{proof}
Cut the manifold $M(\tau_1, \tau_2)$ along $(D^2 \times \gamma) \cap M(\tau_1, \tau_2)= D_n \times \gamma$.
We obtain two $4$-manifolds: $ Q(\tau_1)$ that only depends on $\tau_1$ and $ Q(\tau_2)$ that only depends on $\tau_2$, with
$$M(\tau_1, \tau_2)= Q(\tau_1) \cup _{D_n \times \gamma} Q(\tau_2). $$
With a similar argument as in Lemma \ref{zero}, since $ \partial (D_n \times \gamma)$ is the union of $\partial D^2 \times \gamma$ and some tori whose meridian is also the meridian of a component of the ``missing surface'' $T(\tau_1,\tau_2)$, we can prove that $H_1(\partial (D_n \times \gamma); \C^\omega) = H_1(\partial D^2 \times \gamma; \C^\omega) $, hence the corrective Maslov index term is bounded by the dimension of $H_1(\partial D^2 \times \gamma; \C^\omega),$ which does not depend on the tangles. Therefore
$$\sign_{\omega}(M(\tau_1,\tau_2)) \simeq \sign_{\omega}(Q(\tau_1))+ \sign_{\omega}(Q(\tau_2)). $$
It remains to show that $ \sign_{\omega}(Q(\tau))\simeq 0.$
Thanks to Remark \ref{propertymaslov} we have

$$\sign_{\omega}(Q(\tau))+ \sign_{\omega}(Q(\id_c))\simeq \sign_\omega(M(\tau, \id_c)) \simeq $$ $$\simeq \sigma_\omega(\widehat{\tau})- \sigma_\omega(\widehat{\tau}) - \sigma_\omega(\widehat{\id}_c) - \Maslov( \F_\omega(\overline{\tau}), \Delta, \Delta) = 0 $$
independently of $\tau$. By taking $\tau= \id_c$ we get $ \sign_{\omega}(Q(\id_c))\simeq 0$ and hence $$\sign_{\omega}(Q(\tau)) \simeq 0. $$

This proves the proposition.

\end{proof}
As a consequence
$$  \sigma_{\omega}(\widehat{\tau_1\tau_2}) -  \sigma_{\omega}(\widehat{\tau}_1)- \sigma_{\omega}(\widehat{\tau}_2) -\Maslov(\F_\omega(\overline{\tau}_1), \F_\omega(\id_c), \F_\omega(\tau_2))\simeq \sign_\omega(M(\tau_1,\tau_2)) \simeq 0 $$
for all $(c,c)$-tangles $\tau_1,\tau_2$.

Let $A_j$ be the set of $k \in \{1,\ldots, n\}$ such that $c(k)= \pm j.$
Whenever $\omega\in T^\mu_*$ is such that $I_c(\omega)=  \prod_{j=1}^\mu \omega_j^{i_j} \neq 1 $, where $$i_j= \sum_{k\in A_j} \sgn c(k),$$ we can apply the reduction (Proposition \ref{reduction}) and deduce

$$\sigma_{\omega}(\widehat{\tau_1\tau_2}) -  \sigma_{\omega}(\widehat{\tau}_1)- \sigma_{\omega}(\widehat{\tau}_2)  = \Maslov(\F_\omega(\overline{\tau}_1), \F_\omega(\id_c), \F_\omega(\tau_2)). $$

\subsection{Connections with the formula of Gambaudo and Ghys }
\label{finito}
In the coloured setting there is a generalization of the reduced Burau representation for groups of coloured braids, called reduced coloured Gassner representation. By means of the coloured representation, we are able to state Theorem \ref{main} for coloured braids in a way that clearly generalizes the formula of Gambaudo and Ghys.

\subsubsection{The reduced coloured Gassner representation}
Let $c: \{ 1,\ldots,n\} \to \{\pm 1, \ldots, \pm \mu \}$ be an element of $\Tangles_\mu$. Let $B_c$ denote the group of coloured braids that induce the colouring $c$ on both $D^2 \times\{0\}$ and $D^2 \times\{1\}$, i.e. they are automorphisms of $c$. Let us call $D_n$ the $n$-punctured disk of which we labeled the $j$-th puncture with $c(j)$. The fundamental group $\pi_1(D_n)$ is generated by $x_1,\ldots,x_n$, where $x_j$ is a loop winding one time around the $j$-th puncture counterclockwise when $c(j)$ is positive, and clockwise otherwise. Consider the map
\begin{align*}
    \psi_c: \pi_1(D_n) &\to \Z^\mu = \langle t_1,\ldots, t_\mu \rangle \\
    x_j & \mapsto t_{|c(j)|}
\end{align*}
and let $D_n^\infty \to D_n$ be the regular covering space associated to $\ker(\psi_c)$. The homology groups of $D_n^\infty$ are naturally left modules over $\Lambda_\mu \coloneqq \Z[t^{\pm 1}_1 , \ldots, t^{\pm 1}_\mu]$, in particular we are interested in $H_1(D_n^\infty)$, which however might not be free. Since $D_n$ is homotopically equivalent to the wedge sum of $n$ circles,  one can easily check that
$H_1(D_n^\infty) $ is generated by elements of the form $ (1-t_{|c(i)|})\widetilde{x}_j - (1-t_{|c(j)|})\widetilde{x}_i$ when $c(i) \neq c(j) $ and $\widetilde{x}_j - \widetilde{x}_i $ when $c(i) = c(j) $. 

Fortunately, by means of homology with twisted coefficients, we are able to deal with the existence of torsion.
Let $p:\widetilde{D}_n \to D_n$ be the universal cover of $D_n$. Then $C_*(\widetilde{D}_n)$ is a chain complex of free left $\Z[\pi_1(D_n)]$-modules. The homomorphism $\psi_c$ endows $\Lambda_\mu$ with a right $\Z[\pi_1(D_n)] $-module structure. Consider the twisted homology module $H_1(D_n;\Lambda_\mu)$ and recall that it is isomorphic to $H_1(D_n^\infty)$ (see for example Lemma 5.2.1 \cite{ConwayPhD}). 
Let $S$ be the multiplicative subset of $\Lambda_\mu$ generated by $(1-t_1),\ldots, (1-t_\mu)$ and call $\Lambda_S$ the localization of $\Lambda_\mu$ with respect to $S$. Since localizations are flat, $H_*(D_n; \Lambda_S)$ is isomorphic to $\Lambda_S \otimes_{\Lambda_\mu} H_*(D_n; \Lambda_\mu)$ (it is a consequence of the existence of universal coefficients spectral sequences for twisted coefficients, \cite{ConwayPhD} for details).

The proof of the following can be found in \cite{ConwayPhD}.
\begin{prop}
The $\Lambda_S$-module $H_1(D_n; \Lambda_S)$ is free of rank $n-1$.
\end{prop}

Let $h_\alpha$ be the homeomorphism $D_n \to D_n$ that represents the braid $\alpha$. Then $h_\alpha$ lifts to an automorphism $\widetilde{h}_\alpha : H_1(D_n^\infty) \to H_1(D_n^\infty)$ of $\Lambda_\mu$-modules.

\begin{dfn}
The reduced coloured Gassner representation is the representation $$ \Burau_{(t_1,\ldots,t_\mu)}: B_c \to \Aut_{\Lambda_\mu}(H_1(D_n^{\infty})) $$
which sends $h_\alpha$ to the induced automorphism $\widetilde{h}_\alpha: H_1(D_n^{\infty}) \to H_1(D_n^{\infty}) $. \newline
The localized reduced coloured Gassner representation is the representation
$$ \Burau_{(t_1,\ldots,t_\mu)}^{\text{loc}}: B_c \to \Aut_{\Lambda_S}(\Lambda_S \otimes_{\Lambda_\mu} H_1(D_n^{\infty})) $$
which sends $h_\alpha$ to the induced automorphism $$\id \otimes \widetilde{h}_\alpha: \Lambda_S \otimes_{\Lambda_\mu} H_1(D_n^{\infty}) \to  \Lambda_S \otimes_{\Lambda_\mu} H_1(D_n^{\infty}). $$
\end{dfn}

\begin{rmk}
Given a $(c,c)$-braid $\alpha$, the exterior of the braid deformation retracts on $D_n$. In the notation of the previous section, if $i_0$ is the inclusion of $D_n \times \{0\}$ in  $X_\alpha$ and $ i_1$ is the inclusion of $D_n \times \{1\}$, then the kernel of
$$ k_\alpha= (i_1)_*-(i_0)_*: H_1(D_n;\Lambda_\mu) \oplus H_1 (D_n;\Lambda_\mu)  \to H_*(X_\alpha;\Lambda_\mu) $$
is exactly the graph of the reduced coloured Gassner representation (note that $D_n$ is a deformation retract of $X_\alpha$). The same result holds for $\Lambda_S$-coefficients.
\end{rmk}

Recall that $\mathscr{F}_\omega(\tau)= \ker (j_\tau)$, where 
$$j_\tau: H_1(D_n;\C^{\psi_c, \omega}) \oplus H_1 (D_{n'};\C^{\psi_{c'}, \omega})  \to H_1(X_\tau;\C^\omega), $$
$j_\tau(x,x') = i_{c'}(x')-i_c(x). $
As a consequence of the universal coefficient spectral sequence (see \cite{ConwayPhD}, Proposition 7.5.3), for every CW complex $X$ with a homomorphism $\pi_1(X) \to \Z^\mu= \langle t_1,\ldots,t_\mu \rangle$ such that at least one of the $t_i$ is in its range, for every $\omega \in T^\mu_*$
$$H_1(X; \C^\omega) \cong \C^\omega \otimes_{\Lambda_S} H_1(X; \Lambda_S). $$
Therefore $j_\alpha = \id_{\C^\omega} \otimes_{\Lambda_S} (\id_{\Lambda_S} \otimes_{\Lambda_\mu} k_\alpha)$ for every coloured braid $\alpha$.
We would like to show that $\ker (j_\alpha) = \C^\omega \otimes_{\Lambda_S} \ker (\id_{\Lambda_S} \otimes_{\Lambda_\mu} k_\alpha)$ for certain $\omega\in T^\mu_*$. 

Recall that a submodule of a skew-Hermitian $\Lambda$-module $H$ is called Lagrangian if it is equal to its annihilator.
\begin{lem}
For $\tau$ a $(c,c')$-tangle and for $\omega$ such that $I_c(\omega)$ and $ I_{c'}(\omega)$ are not equal to $1$, the subspace $\mathscr{F}_\omega(\tau) $ is Lagrangian.
\end{lem}

\begin{proof}
With the notation of the proof of Theorem \ref{functor} we just need to show that $\ker(\varphi \circ k)$ is Lagrangian, where $k$ is the map induced by the inclusion of $D_n \times \{ 0 \} \sqcup D_n \times \{ 1 \}$ in $\partial X_\tau$ and $\varphi$ is the map induced by the inclusion of the boundary in $X_\tau$. The subspace $\ker \varphi$ is Lagrangian because it is the kernel of the inclusion of the boundary in a 3-manifold (this is a standard consequence of the Poincaré duality).
 Using the fact that $\partial X_\tau \setminus( D_n \sqcup D_n')$ is a union of cylinders and tori with vanishing twisted homology modules (similarly to the proof of Lemma \ref{nondeg} and Lemma \ref{zero}), a standard Mayer-Vietoris argument shows that $k$ is an isomorphism. This concludes the proof.

\end{proof}

Since the form on $(-H_1(D_n;\C^{\psi_c, \omega})) \oplus H_1 (D_{n'};\C^{\psi_{c'}, \omega}) $ is skew-Hermitian and non-degenerate the dimension of $\mathscr{F}_\omega(\tau)  $ is half the dimension of $$ H_1(D_n;\C^{\psi_c, \omega}) \oplus H_1 (D_{n'};\C^{\psi_{c'}, \omega})$$ which is equal to $ (n-1)+(n'-1). $ 

In \cite{CimasoniTuraev} Cimasoni and Turaev show that the kernel of the map $k_\alpha$ is Lagrangian in the $\Lambda_\mu$-module $H_1(D_n;\Lambda_\mu) \oplus H_1 (D_n;\Lambda_\mu)$, which is endowed with a non degenerate skew-Hermitian form, and it can be seen that its localization is free over $\Lambda_S$ and therefore the localization has rank $n-1$. As a consequence, since obviously $$ \ker (j_\alpha) \supset \C^\omega \otimes_{\Lambda_S} \ker (\id_{\Lambda_S} \otimes_{\Lambda_\mu} k_\alpha),$$ and the two spaces have the same complex dimension, they are equal.

As a consequence, $\mathscr{F}_\omega(\alpha) $ is equal to $\Gamma_{\Burau_{\omega}(\alpha)}$ the graph of the reduced coloured Gassner representation tensorized by $\C^\omega$ and therefore it can be interpreted as the graph of the evaluation in $t=\omega$ of a matrix representing the localized coloured Gassner representation. 

The \emph{Meyer cocycle} of two unitary automorphisms $\gamma_1,\gamma_2$ of a skew-Hermitian complex vector space can be defined considering the space
$$E_{\gamma_1,\gamma_2} = \Imm(\gamma_1^{-1} -\id) \cap \Imm( \id - \gamma_2) $$
and defining on it a Hermitian form
$$b: E_{\gamma_1,\gamma_2} \times E_{\gamma_1,\gamma_2} \to \C$$
by setting $b(x,y) = \lambda (x_1+x_2, y)$ for $x= \gamma_1^{-1}(x_1)-x_1=x_2-\gamma_2(x_2) \in E_{\gamma_1,\gamma_2}.$
Then $\Meyer(\gamma_1, \gamma_2)$ is defined as the signature of $b$.

The following lemma is easy to prove:
\begin{lem}
\label{Meyer}
$$\Meyer(\gamma_1,\gamma_2)= - \Maslov(\Gamma_{\gamma_1^{-1}}, \Gamma_{\id}, \Gamma_{\gamma_2}) $$
where $\Gamma_\gamma$ indicates the graph of $\gamma$.
\end{lem}

Now Lemma \ref{Meyer} allows us to conclude:

\begin{thm}
For every $\alpha, \beta \in B_c$ coloured braids and for every $\omega \in T^\mu_*$ such that $I_c(\omega) \neq 1$
\begin{equation*} 
\sigma_\omega(\widehat{\alpha \beta}) -\sigma_\omega(\widehat{\alpha})-\sigma_\omega(\widehat{\beta}) = -\Meyer(\Burau_{\omega}(\alpha), \Burau_{\omega}(\beta)).
\end{equation*}
\end{thm}

This result allows us to compute some real examples and show that the hypothesis of Theorem \ref{main} cannot be weakened. The following examples were studied by Cimasoni and Conway \cite{CimasoniConway} and showed that their formula could not be extended too much. We will see that in these two cases our formulation of the theorem cannot be extended further.

\begin{ex}
Consider the case $\mu=1$, $n=2$ and let $c:\{1,2\} \to \{\pm 1\}$ be constantly equal to $+1$. Notice that in this situation $I_c(\omega)= \omega^2$. Let $\alpha$ be the standard positive generator of the braid group $B_2$. The link $\widehat{\alpha \alpha}$ is the positive Hopf link $\mathcal{H}$, while $\widehat{\alpha}$ is the unknot. 
Of course $\sigma_\omega(\widehat{\alpha})=0$ because the unknot bounds a disk. 
The positive Hopf link bounds a surface $S$ and $H_1(S) \cong \Z$ is generated by the blue curve in Figure \ref{Hopf}, which we call $\gamma$. Notice that $\lk(\gamma, i_+(\gamma))= -1$, hence a matrix representing the Seifert form for $\mathcal{H}$ is $$ A=(-1). $$ 
As a consequence $(1-\omega) A + (1 - \overline{\omega})A^T=(-2+2 \Re \omega) $ and the signature $\sigma_\omega(\mathcal{H})$ is equal to $-1$ for every $\omega \in S^1 \setminus \{1\}$. As it is shown by Cimasoni and Conway in \cite{CimasoniConway},
$$\Meyer(\Burau_\omega(\alpha), \Burau_\omega(\alpha))= \begin{cases} 1 & \text{if } \omega \neq -1, \, \omega \in S^1 \setminus\{1\}; \\ 0 & \text{if } \omega = -1.  \end{cases} $$
Therefore the equality is satisfied for all $\omega \in S^1 \setminus \{\pm 1\}$ but not for $\omega = -1$ and in fact $I_c(-1)=1$. This shows that even in the basic case with $\mu=1$ one needs the assumption $I_c(\omega) \neq 1 $ for the equality in Theorem \ref{main} to hold.

\begin{figure}[H]
    \centering
    \includegraphics[width = 6 cm]{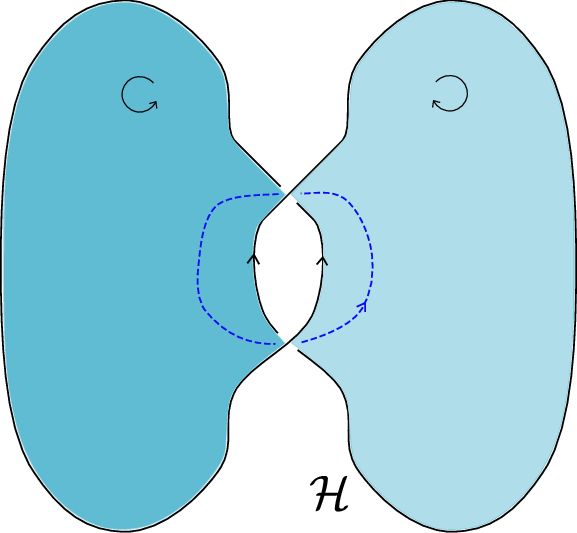}
    \caption{A Seifert surface for the Hopf link.}
    \label{Hopf}
\end{figure}
\end{ex}

\begin{ex}
Consider the case $\mu=2$, $n=2$ and let $c:\{1,2\} \to \{\pm 1, \pm 2\}$ be the identity. Notice that for $\omega=(\omega_1,\omega_2) \in (S^1 \setminus \{1\})^\mu$ we have $I_c(\omega)= \omega_1 \omega_2$. The $2$-coloured link $L$ in Figure \ref{pure} is equal to $\widehat{\beta\beta}$, where $\beta$ is the $(c,c)$-braid in Figure \ref{beta}. The link $\widehat{\beta}$ is the $2$-coloured Hopf link. There is a $C$-complex for the $2$-coloured Hopf link which consists of two disks that intersect transversely in one clasp intersection and this $C$-complex is contractible. Hence the multivariate signature for the $2$-coloured Hopf link vanishes for every $\omega \in (S^1 \setminus \{1\})^2.$ As it is shown by Cimasoni and Conway in \cite{CimasoniConway} 
$$\sigma_\omega(L)= -\sgn(\Re((1-\omega_1)(1-\omega_2))), $$
while 
$$\Meyer(\Burau_\omega(\beta), \Burau_\omega(\beta))= -\sgn(\Re((1-\omega_1)(1-\omega_2)(1-\omega_1\omega_2))). $$
It is possible to see that these quantities have the same sign if and only if $\omega_1\omega_2 \neq 1$, therefore the equality in Theorem \ref{main} holds if and only if $I_c(\omega) \neq 1$.

\begin{figure}[H]
    \centering
    \includegraphics[width = 9 cm]{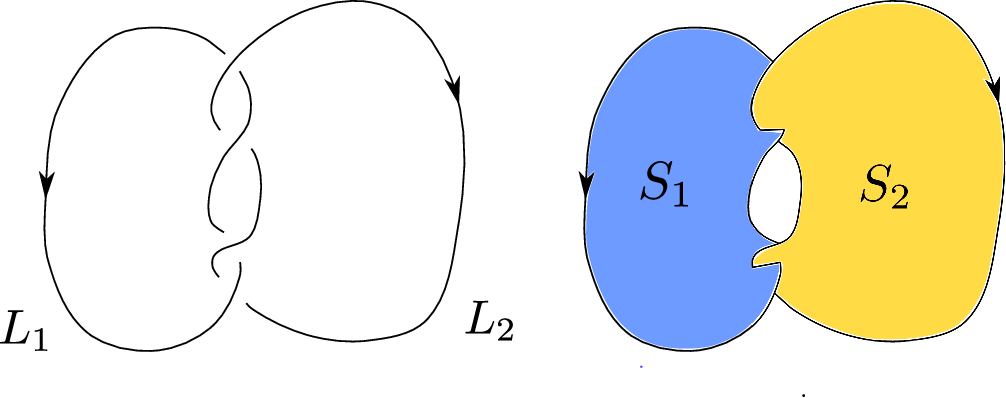}
    \caption{The $2$-coloured link $L$.}
    \label{pure}
\end{figure}

\begin{figure}[H]
    \centering
    \includegraphics[width = 4 cm] {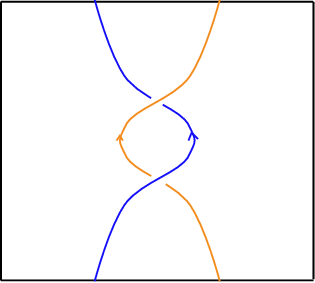}
    \caption{The coloured braid $\beta$.}
    \label{beta}
\end{figure}
\end{ex}

\end{document}